\documentclass[11pt]{amsart}
\usepackage{amsfonts}
\usepackage{amsmath}
\usepackage{amssymb}
\usepackage{graphicx}

\theoremstyle{definition}
\newtheorem{definition}{Definition}[section]
\theoremstyle{plain}
\newtheorem{lemma}[definition]{Lemma}
\newtheorem{proposition}[definition]{Proposition}
\newtheorem{theorem}[definition]{Theorem}
\newtheorem{corollary}[definition]{Corollay}
\numberwithin{equation}{section}
\input{tcilatex}

\begin{document}
\title[Existence and Convergence Theorems]{Existence and Convergence
Theorems for Multivalued Generalized Hybrid mappings in CAT($\kappa $)-spaces%
}
\author[E. Haci\u{g}lu]{Emirhan Haci\u{g}lu}
\address[E. HACIO\u{G}LU]{Department of Mathematics, Yildiz Technical
University, Davutpasa Campus, Esenler, 34220 Istanbul, Turkey}
\email{\texttt{emirhanhacioglu@hotmail.com}}
\author[V. Karakaya]{Vatan Karakaya}
\address[V. Karakaya]{Department of Mathematical Engineering, Yildiz
Technical University, Davutpasa Campus, Esenler, 34210 Istanbul,Turkey}
\email{\texttt{vkkaya@yahoo.com}}
\subjclass[2010]{????????}

\begin{abstract}
In this study we give definition of some multivalued hybrid mappings which
are general than multivalued nonexpansive mappings and some others. Also we
give existence and convergence results in CAT($\kappa $)-spaces
\end{abstract}

\maketitle

\section{Introduction}

Let $K$ be a nonempty subset of a Hilbert space $H$ and $T:K\rightarrow H$
be a mapping then, for all $x,y\in K,$ if $T$ satisfies 
\begin{equation*}
||Tx-Tx||\leq ||x-y||,
\end{equation*}%
\begin{equation*}
2||Tx-Tx||^{2}\leq ||Tx-y||^{2}+||Ty-x||^{2},
\end{equation*}%
\ and%
\begin{equation*}
3||Tx-Tx||^{2}\leq ||x-y||^{2}+||Tx-y||^{2}+||Ty-x||^{2}
\end{equation*}%
then it called nonexpansive, nonspreading\cite{kohsaka} and hybrid\cite%
{takhas} respectively and none of this mappings included in other, In 2010
Aoyama et al.\cite{aoyam} defined $\lambda -$hybrid as a follows%
\begin{equation*}
(1+\lambda )||Tx-Tx||^{2}-\lambda ||x-Ty||^{2}\leq (1-\lambda
)||x-y||^{2}+\lambda ||Tx-y||^{2}
\end{equation*}%
where $\lambda \in 
\mathbb{R}
$. $\lambda -$hybrid mappings are general than nonexpansive mappings,
nonspreading mappings and also hybrid mappings. In 2011 Aoyama and Kohsaka%
\cite{aoyam2} introduced $\alpha -$nonexpansive mappings in Banach spaces as
follows,%
\begin{equation*}
||Tx-Tx||^{2}\leq (1-2\alpha )||x-y||^{2}+\alpha ||Tx-y||^{2}+\alpha
||x-Ty||^{2}
\end{equation*}%
where $\alpha <1$ and also they showed that $\alpha -$nonexpansive and $%
\lambda -$hybrid are equivalent in Hilbert spaces for $\lambda <2$. Kocourek
et al.\cite{kocour} introduced more general mapping class than above
mappings in Hilbert spaces,called $(\alpha ,\beta )-$generalized hybrid, as
follows 
\begin{equation*}
\alpha ||Tx-Tx||^{2}+(1-\alpha )||x-Ty||^{2}\leq \beta ||Tx-y||^{2}+(1-\beta
)||x-y||^{2}
\end{equation*}%
Also, in 2011, Lin et al.\cite{lin}defined a generalized hybrid mappings in
CAT$(0)$, which is more general than nonexpansive, nonspreading and hybrid
mappings in Banach spaces$,$ as a follows%
\begin{eqnarray*}
d^{2}(Tx,Ty)\leq a_{1}(x)d^{2}(x,y)+a_{2}(x)d^{2}(Tx,y)+a_{3}(x)d^{2}(x,Ty)
\\
+k_{1}(x)d^{2}(Tx,x)+k_{2}(x)d^{2}(Ty,y)
\end{eqnarray*}%
where $a_{1,}a_{2},a_{3},k_{1},k_{2}:X\rightarrow \lbrack 0,1]$ with $%
a_{1}(x)+a_{2}(x)+a_{3}(x)<1$, $2k_{1}(x)<1-a_{2}(x)$ and $%
2k_{2}(x)<1-a_{3}(x)$ for all $x,y\in X.$It is easy to see that $(\alpha
,\beta )-$generalized hybrid and generalized hybrid mappings are independent
in metric spaces. In this paper we define two multivalued mapping class
which general than $(\alpha ,\beta )-$generalized hybrid and generalized
hybrid mappings, then establish existence and convergence theorems in Cat($%
\kappa $) spaces for $\kappa >0$.

\section{Preliminaries}

\qquad\ Let $(X,d)$ be a metric space and $K$ a nonempty subset of $X$ . If
there is $y\in K$ such that $d(x,y)=d(x,B)=\inf \{d(x,y);y\in B\}$ then $K$
is called proximinal subset of $X$. The family of nonempty compact\ convex
subsets of $X$ , the family of nonempty closed and convex subsets of $X,$the
family of nonempty closed and bounded subsets of $X$ and the family of
nonempty proximinal subsets will be denoted by $KC(X),$ $CC(X),$ $CB(X),P(X)$
, respectively. Let $H$ Haussdorf Metric on $CB(X),$ defined by

\begin{equation*}
H(A,B)=\max \{\sup_{x\in A}d(x,B),\sup_{x\in B}d(x,A)\}
\end{equation*}%
\ where $d(x,B)=\inf \{d(x,y);y\in B\}.$

\qquad A multivalued mappings $T:K\rightarrow 2^{E}$ is called nonexpansive
if for all $x,y\in K$ \ and $p\in F(T)$\ 
\begin{equation*}
H(Tx,Ty)\leq d(x,y)
\end{equation*}

A point \ is called fixed point of $T$ \ if $x\in Tx$ and the set of all
fixed points of $T$ is denoted by $F(T).$

\qquad Many iterative method to approximate a fixed points of the mappings
have been introduced for single-valued mappings in Banach spaces, well known
one is defined by Picard as%
\begin{equation*}
x_{n+1}=Tx_{n}
\end{equation*}%
After Picard, Mann and Ishikawa introduced new iteration procedures,
respectively, as follows%
\begin{equation*}
x_{n+1}=(1-\alpha _{n})x_{n}+\alpha _{n}Tx_{n}
\end{equation*}%
and%
\begin{eqnarray*}
x_{n+1} &=&(1-\alpha _{n})x_{n}+\alpha _{n}Ty_{n} \\
y_{n} &=&(1-\beta _{n})x_{n}+\beta _{n}Tx_{n}
\end{eqnarray*}%
where $\{\alpha _{n}\}$ and $\{\beta _{n}\}$are $\ $sequences\ in$\ [0,1].$%
In 2007, Agarwal \textit{et al. }defined following iteration as%
\begin{eqnarray*}
s_{n+1} &=&(1-\alpha _{n})Ts_{n}+\alpha _{n}Tt_{n} \\
t_{n} &=&(1-\beta _{n})s_{n}+\beta _{n}Ts_{n}
\end{eqnarray*}%
where $\{\alpha _{n}\}$ and $\{\beta _{n}\}$are $\ $sequences\ in$\ [0,1].$%
In 2008, \ S. Thianwan introduce two step iteration as%
\begin{eqnarray*}
x_{n+1} &=&(1-\alpha _{n})y_{n}+\alpha _{n}Ty_{n} \\
y_{n} &=&(1-\beta _{n})x_{n}+\beta _{n}Tx_{n}
\end{eqnarray*}%
where $\{\alpha _{n}\}$ and $\{\beta _{n}\}$are $\ $sequences\ in$\ [0,1].$%
M.A. Noor defined Noor iteration in 2001 and Phuengrattana and Suantai
defined the SP iteration as follows%
\begin{eqnarray*}
x_{n+1} &=&(1-\alpha _{n})x_{n}+\alpha _{n}Ty_{n} \\
y_{n} &=&(1-\beta _{n})x_{n}+\beta _{n}Tz_{n} \\
z_{n} &=&(1-\gamma _{n})x_{n}+\gamma _{n}Tx_{n}
\end{eqnarray*}%
and%
\begin{eqnarray*}
x_{n+1} &=&(1-\alpha _{n})y_{n}+\alpha _{n}Ty_{n} \\
y_{n} &=&(1-\beta _{n})z_{n}+\beta _{n}Tz_{n} \\
z_{n} &=&(1-\gamma _{n})x_{n}+\gamma _{n}Tx_{n}
\end{eqnarray*}%
where $\{\alpha _{n}\}$, $\{\beta _{n}\}$ and \ $\{\gamma _{n}\}$ are $\ $%
sequences\ in$\ [0,1].$ Recently Renu Chugh, Vivek Kumar and Sanjay Kumar
defined CR- iteration as follows%
\begin{eqnarray*}
x_{n+1} &=&(1-\alpha _{n})y_{n}+\alpha _{n}Ty_{n} \\
y_{n} &=&(1-\beta _{n})Tx_{n}+\beta _{n}Tz_{n} \\
z_{n} &=&(1-\gamma _{n})x_{n}+\gamma _{n}Tx_{n}
\end{eqnarray*}%
where $\{\alpha _{n}\}$, $\{\beta _{n}\}$ and\ $\{\gamma _{n}\}$ are
sequences\ in$\ [0,1]$ with $\{\alpha _{n}\}$ satisfying $\sum_{n=0}^{\infty
}\alpha _{n}=\infty .$

Also Gursoy and Karakaya [8] introduced Picard-S iteraion as follows%
\begin{eqnarray*}
x_{n+1} &=&Ty_{n} \\
y_{n} &=&(1-\alpha _{n})Tx_{n}\oplus \alpha _{n}Tz_{n} \\
z_{n} &=&(1-\beta _{n})x_{n}\oplus \beta _{n}Tx_{n}
\end{eqnarray*}%
where $\{\alpha _{n}\}$ and $\{\beta _{n}\}$ are$\ $sequences\ in$\ [0,1].$

The multivalued version of Thianwan iteration and Picad-S iteration defined
as a follows

\begin{eqnarray}
x_{n+1} &=&P_{K}((1-\alpha _{n})y_{n}\oplus \alpha _{n}u_{n})  \TCItag{1.1}
\\
y_{n} &=&P_{K}((1-\beta _{n})x_{n}\oplus \beta _{n}v_{n})  \notag
\end{eqnarray}%
where $\{\alpha _{n}\}$ and $\{\beta _{n}\}$are sequences\ in$\ [0,1]$ and $%
u_{n}\in Ty_{n},$ $v_{n}\in Tx_{n}$ and%
\begin{eqnarray}
x_{n+1} &=&P_{K}(u_{n})  \TCItag{1.2} \\
y_{n} &=&P_{K}((1-\alpha _{n})w_{n}\oplus \alpha _{n}v_{n})  \notag \\
z_{n} &=&P_{K}((1-\beta _{n})x_{n}\oplus \beta _{n}w_{n})  \notag
\end{eqnarray}%
where $\{\alpha _{n}\}$ is a$\ $sequences\ in$\ [0,1],$ $u_{n}\in Ty_{n},$ $%
v_{n}\in Tz_{n},w_{n}\in Tx_{n}$.

\qquad Before the results we give some definitions and lemmas about $%
CAT(\kappa )$ with $\kappa >0$ and \ $\Delta -$convergences.

\qquad Let $(X,d)$ bounded metric space, $x,y\in X$ and $C\subseteq X$
nonempty subset. A geodesic path (or shortly a geodesic) joining x and y is
a map $c:[0,t]\subseteq 
\mathbb{R}
\rightarrow X$ \ such that $c(0)=x$, $c(t)=y$ and $d(c(r),c(s))=|r-s|$ for
all $r,s\in \lbrack 0,t].$In particular $c$ is an isometry and $%
d(c(0),c(t))=t.$ The image of $c,$ $c([o,t])$ is called geodesic segment
from $x$ to $y$ and it it is unique (it not necessarily be unique) then it
is denoted by $[x,y].$ $z\in \lbrack x,y]$ if and only if for an $t\in
\lbrack 0,1]$ such that $d(z,x)=(1-t)d(x,y)$ and $d(z,y)=td(x,y).$ The point 
$z$ is dednoted by $z=(1-t)x\oplus ty.$\ For fixed $r>0,$ the space $(X,d)$
\ is called r-geodesic space if any two point $x,y\in X$ with $d(x,y)<r$ \
there is a geodesic joining $x$ to $y$. if for every $x,y\in X$ there is a
geodesic path then $(X,d)$ called geodesic space and uniquely geodesic space
if that geodesic path is unique for any pair $x,y.$ A subset $C\subseteq X$
is called convex if it contains all geodesic segment joining any pair of
points in it.

\begin{definition}
\label{(def.2.1)} \cite{brid}Given a real number $\kappa $ then

\begin{enumerate}
\item[i)] i\textit{f }$\kappa =0$\textit{\ then }$M_{\kappa }^{n}$\textit{\
is Euclidean space }$E^{n}$

\item[ii)] \textit{if }$\kappa >0$\textit{\ then }$M_{\kappa }^{n}$\textit{%
is obtained from the shpere }$S^{n}$\textit{\ by multiplying distance
function by }$\frac{1}{\sqrt{\kappa }}$

\item[iii)] \textit{if }$\kappa <0$\textit{\ then }$M_{\kappa }^{n}$\textit{%
\ is obtained from hyperbolic space }$H^{n}$\textit{\ by multiplying
distance function by }$\frac{1}{\sqrt{-\kappa }}$
\end{enumerate}
\end{definition}

\qquad In geodesic metric space $(X,d)$ , A geodesic triangle $\Delta
(x,y,z) $ \ consist of three point $x,y,z$ as vertices and three geodesic
segments of any pair of these points,that is, $q\in $ $\Delta (x,y,z)$ means
that $q\in \lbrack x,y]\cup \lbrack x,z]\cup \lbrack y,z].$ The triangle $%
\overline{\Delta }(\overline{x},\overline{y},\overline{z})$ in $M_{\kappa
}^{2}$ is called comparison triangle for the triangle $\Delta (x,y,z)$ such
that $d(x,y)=d(\overline{x},\overline{y}),d(x,z)=d(\overline{x},\overline{z}%
) $ and $d(y,z)=d(\overline{y},\overline{z})$ and such a comparison triangle
always exist provided that the perimeter $d(x,y)+d(y,z)+d(z,x)<2D_{\kappa }$%
( $D_{\kappa }=\frac{\pi }{\sqrt{K}}$ if $\kappa >0$ and $\infty $
otherwise) in $M_{\kappa }^{2}$ $\kappa $ 2.14 in \cite{brid}. A point point 
$\overline{z}\in \lbrack \overline{x},\overline{y}]$ called comparison point
for $z\in \lbrack x,y]$ if $d(x,z)=d(\overline{x},\overline{z}).$ A geodesic
triangle $\Delta (x,y,z)$ in $X$ with perimeter less than $2D_{\kappa }$
(and given a comparison triangle $\overline{\Delta }(\overline{x},\overline{y%
},\overline{z})$ for $\Delta (x,y,z)$ in $M_{\kappa }^{2}$) satisfies $%
CAT(\kappa )$ inequality if$\ d(p,q)\leq d(\overline{p},\overline{q})$ for
all $p,q\in $ $\Delta (x,y,z)$ where $\overline{p},\overline{q}\in \overline{%
\Delta }(\overline{x},\overline{y},\overline{z})$ are the comparison points
of $p,q$ respectively. The $D_{\kappa }$-geodesic metric space $(X,d)$ is
called $CAT(\kappa )$ space if \ every geodesic triangle in $X$ with
perimeter less than $2D_{\kappa }$ satisfies the $CAT(\kappa )$\ inequality.

Bruhat and Tits \cite{brid} shows that If $(X,d)$ \ is a $CAT(0)$ space,

\begin{equation*}
d^{2}(x,\frac{1}{2}y\oplus \frac{1}{2}z)\leq \frac{1}{2}d^{2}(x,y)+\frac{1}{2%
}d^{2}(x,z)-\frac{1}{4}d^{2}(y,z)
\end{equation*}%
satisfied for every $x,y,z\in X$, its called \textit{CN inequality} \ and
its generalized by Dhompongsa and Panyanak \cite{brid} as follows

\begin{equation*}
d^{2}(x,(1-\lambda )y\oplus \lambda z)\leq (1-\lambda )d^{2}(x,y)+\lambda
d^{2}(x,z)-\lambda (1-\lambda )d^{2}(y,z)
\end{equation*}%
for every $x,y,z\in X,\lambda \in \lbrack 0,1]$ (\textit{CN}$^{\ast }$%
\textit{\ inequality} ).

In fact, for a geodesic metric space $(X,d)$ following three statements are
equivalent;

\begin{enumerate}
\item[i)] $(X,d)$ is a $CAT(0)$

\item[ii)] $(X,d)$ satisfied \textit{CN inequality }

\item[iii)] $(X,d)$ satisfied \textit{CN}$^{\ast }$\textit{\ inequality }
\end{enumerate}

Let $(X,d)$ be geodesic space and $R\in (0,2]$. if for every $x,y,z\in X$

\begin{equation*}
d^{2}(x,(1-\lambda )y\oplus \lambda z)\leq (1-\lambda )d^{2}(x,y)+\lambda
d^{2}(x,z)-\frac{R}{2}\lambda (1-\lambda )d^{2}(y,z)
\end{equation*}%
is satisfied then $(X,d)$ called $R-$\textit{convex} \cite{ohta} Hence,\ $%
(X,d)$ is a $CAT(0)$ space if and only if it is a $2-$\textit{convex space}

\begin{proposition}
\cite{rash} The modulus of convexity for $CAT(\kappa )$ space $X$ (of
dimension $\geq 2$) and number $r<$ $\frac{\pi }{2\sqrt{\kappa }}$ and let $%
m $ denote the midpoint of the segment $[x,y]$ joining $x$ and $y$ define by
the modulus $\delta _{r}$ by sitting
\end{proposition}

\begin{equation*}
\delta (r,\in )=\inf \{1-\frac{1}{r}d(a,m)\}\medskip
\end{equation*}%
where the infimum is taken over all points $a,x,y\in X$ satisfying $%
d(a,x)\leq r$, $d(a,y)\leq r$ and $\in \leq d(x,y)<\frac{\pi }{2\sqrt{\kappa 
}}$

\begin{lemma}
\label{(lemma2.22)}\cite{lao} Let $X$ be a complete $CAT(\kappa )$space with
modulus of convexity $\delta (r,\in )$ and let $x\in E$. Suppose that $%
\delta (r,\in )$ increases with $r$ (for a fixed $\in $ ) and suppose $%
\{t_{n}\}$ is a sequence in $[b,c]$ for some $b,c\in (0,1)$, $\{x_{n}\}$ and 
$\{y_{n}\}$ are the sequences in $X$ such that $\limsup\nolimits_{n%
\rightarrow \infty }d(x_{n},x)\leq r,$ $\limsup\nolimits_{n\rightarrow
\infty }d(y_{n},x)\leq r$ and $\lim\nolimits_{n\rightarrow \infty
}d((1-t_{n})x_{n}\oplus t_{n}y_{n},x)=r$ for some $r\geq 0$. Then $%
\lim_{n\rightarrow \infty }d(x_{n},y_{n})=0.$
\end{lemma}

\begin{lemma}
\label{(lemma2.2)}\cite{panyanac} Let $\kappa $ be an arbitrary positive
real number and $(X,d)$ be a $CAT(\kappa )$ space with $diam(X)<$ $\frac{\pi
-\varepsilon }{2\sqrt{\kappa }}$ for some $\varepsilon \in (0,\frac{\pi }{2}%
).$Then $(X,d)$ is a $R-$\textit{convex space for} $R=(\pi -2\varepsilon
)\tan (\varepsilon ).$
\end{lemma}

\begin{proposition}
\label{(prop2.3)} \cite{brid} $M_{\kappa }^{n}$ is a geodesic metric space.
If $\kappa \leq 0$, then $M_{\kappa }^{n}$ is uniquely geodesic and all
balls in $M_{\kappa }^{n}$ are convex. If \ $\kappa >0$, then there is a
unique geodesic segment joining x, y in $M_{\kappa }^{n}$ \ if and only if $%
d(x;y)<\frac{\pi }{\sqrt{\kappa }}$. If $\kappa >0$ , closed balls in $%
M_{\kappa }^{n}$ of radius smaller than $\frac{\pi }{2\sqrt{\kappa }}$ are
convex.
\end{proposition}

\begin{proposition}
\label{(prop2.4)} \cite{brid} Let $X$ be $CAT(\kappa )$ space. Then any ball
of radius smaller than $\frac{\pi }{2\sqrt{\kappa }}$ are convex.
\end{proposition}

\begin{proposition}
\label{(prop2.5)} Exercise 2.3(1) in \cite{brid} Let $\kappa >0$ and $(X,d)$
be a $CAT(\kappa )$ space with $diam(X)<\frac{D_{\kappa }}{2}=\frac{\pi }{2%
\sqrt{\kappa }}$ Then, for any$x,y,z\in X$ and $t\in \lbrack 0,1]$, we have%
\begin{equation*}
d((1-t)x\oplus ty,z)\leq (1-t)d(x,z)+td(y,z).
\end{equation*}
\end{proposition}

\begin{proposition}
\label{(prop2.6)}$%
\mathbb{R}
$-trees are particular class of $CAT(\kappa )$ spaces for any real number $%
\kappa $ (see p.167 in \cite{brid}) and family of closed convex subsets of a 
$CAT(\kappa )$ spaces has uniform normal structure in usual metric sense.

\begin{definition}
\label{(def2.7)}\cite{brid} An $%
\mathbb{R}
$-tree is a metric space $X$ such that
\end{definition}
\end{proposition}

\begin{enumerate}
\item[i)] it is a uniquely geodesic metric space,

\item[ii)] if $x,y,z\in X$ are such that $[y,x]\cap \lbrack x;z]=\{x\},$
then $[y,x]\cup \lbrack x,z]=[y,z].$
\end{enumerate}

Let $\{x_{n}\}$ be a bounded sequence in a $CAT(\kappa )$ space $X$ and $%
x\in X$. $\ $Then, with setting%
\begin{equation*}
r(x,\{x_{n}\})=\limsup_{n\rightarrow \infty }d(x,x_{n})
\end{equation*}

the asymptotic radius of $\{x_{n}\}$ is defined by%
\begin{equation*}
r(\{x_{n}\})=\inf \{r(x,\{x_{n}\});x\in X.\},
\end{equation*}

the asymptotic radius of $\{x_{n}\}$ with respect to $K\subseteq X$ is
defined by%
\begin{equation*}
r_{K}(\{x_{n}\})=\inf \{r(x,\{x_{n}\});x\in K.\}
\end{equation*}

and the asymptotic center of $\{x_{n}\}$ is defined by%
\begin{equation*}
A(\{x_{n}\})=\{x\in X:r(x,\{x_{n}\})=r(\{x_{n}\})\}.
\end{equation*}

and let $\omega _{w}(x_{n}):=\cup A(\{x_{n}\})$ where union is taken on all
subsequences of $\{x_{n}\}.$

\begin{definition}
\label{(def2.8)}\cite{esp}A sequence $\{x_{n}\}\subset X$ \ is said to be $%
\Delta -$ convergent to $x\in X$ if \ $x$ is the unique asymptotic center of
\ all subsequence $\{u_{n}\}$ of $\{x_{n}\}$. In this case we write $\Delta
-\lim_{n}x_{n}=x$ and read as x is $\Delta -$limit of $\{x_{n}\}.$
\end{definition}

\begin{proposition}
\label{(prop2.9)}\cite{esp}Let $X$ be a complete $CAT(\kappa )$ space, $%
K\subseteq X$ nonempty, closed and convex, $\{x_{n}\}$ is a sequence in $X$.
If $r_{C}(\{x_{n}\})$ $<$ $\frac{\pi }{2\sqrt{\kappa }}$ then $%
A_{C}(\{x_{n}\})$ consist exactly one point.
\end{proposition}

\begin{lemma}
\label{(lemma2.10)}\cite{dogh2}
\end{lemma}

\begin{enumerate}
\item[i)] \textit{Every bounded sequence in }$X$\textit{\ has a }$\Delta $%
\textit{-convergent subsequence}

\item[ii)] \textit{If }$K$\textit{\ is a closed convex subset of }$X$\textit{%
\ and if }$\{x_{n}\}$\textit{\ is a bounded sequence in }$K$\textit{, then
the\ asymptotic center of }$\{x_{n}\}$\textit{\ is in }$K$

\item[iii)] \textit{If }$K$\textit{\ is a closed convex subset of }$X$%
\textit{\ and if }$f$\textit{\ }$:K$\textit{\ }$\rightarrow X$\textit{\ is a
nonexpansive mapping, then the conditions, }$\{x_{n}\}$\textit{\ }$\Delta $%
\textit{-converges to }$\mathit{x}$\textit{\ and }$\lim_{n\rightarrow \infty
}d(x_{n},f(x_{n}))=0$\textit{, imply }$f(x)=x$\textit{\ and }$x\in K.$
\end{enumerate}

\begin{lemma}
\label{(lemma2.11)}\cite{dogh2} If $\{x_{n}\}$ is a bounded sequence in $X$
with $A(\{x_{n}\})=\{x\}$ and $\{u_{n}\}$ is a subsequence of $\{x_{n}\}$
with $A(\{u_{n}\})=u$ and the sequence $\{d(x_{n},u)\}$ converges, then $x=u$
\end{lemma}

\begin{definition}
A multivalued mapping $T:K\rightarrow CB(K)$ is said to satisfy Condition
(I)\ if there is a nondecreasing function $f:[0,\infty ]\rightarrow \lbrack
0,\infty ]$ $f(0)=0,$ $f(r)>0$ for all $r\in (0,\infty )$ such that $%
d(x,Tx)\geq f(d(x,F))$ for all $x\in K$ where $F=F(T).$
\end{definition}

\begin{definition}
The mapping $T:X\rightarrow CB(X)$ is called hemicompact if, for any
sequence $\{x_{n}\}\subset X$ \ such that $d(x_{n},Tx_{n})$ $\rightarrow 0$
as $n$ $\rightarrow \infty $, there exists a subsequence $\{x_{n_{k}}\}$ of $%
\{x_{n}\}$ such that $x_{n_{k}}\rightarrow p\in X$
\end{definition}

\begin{lemma}
\label{(lemma2.12)}\cite{esp} )Let $\kappa >0$ and $X$ be a complete $%
CAT(\kappa )$ space with $diam(X)\leq \frac{\pi -\varepsilon }{2\sqrt{\kappa 
}}$ for some $\varepsilon \in (0,\pi /2)$. Let $K$ be a nonempty closed
convex subset of $X$. Then
\end{lemma}

\begin{enumerate}
\item[i)] the metric projection $P_{K}(x)$ of $x$onto $K$ is a singleton,

\item[ii)] if $x\notin K$ and $y\in K$ with $u\neq P_{K}(x)$, then $\angle
_{P_{K}(x)}(x,y)\geq \frac{\pi }{2}$,

\item[iii)] for each $y\in K$, $d(P_{K}(x),P_{K}(y))\leq d(x,y)$.
\end{enumerate}

\begin{definition}
$T$ is called $(a_{1},a_{2},b_{1},b_{2})-$multivalued hybrid mapping type I
\ from $X$ to $CB(X)$ if 
\begin{equation*}
a_{1}(x)d^{2}(u,v)+a_{2}(x)d^{2}(u,y)\leq
b_{1}(x)d^{2}(x,v)+b_{2}(x)d^{2}(x,y)
\end{equation*}%
satisfied for all $x,y\in X,$ $u\in Tx$ and $v\in Ty$ where $%
a_{1},a_{2}:X\rightarrow 
\mathbb{R}
\backslash (0,1)$ and $b_{1},b_{2}:X\rightarrow \lbrack 0,1]$ with $%
a_{1}(x)+a_{2}(x)\geq 1$ and $b_{1}(x)+b_{2}(x)\leq 1.$
\end{definition}

\begin{definition}
$T$ is called generalized multivalued hybrid mapping type I\ from $X$ to $%
CB(X)$ \ if 
\begin{eqnarray*}
d^{2}(u,v) &\leq &a_{1}(x)d^{2}(x,y)+a_{2}(x)d^{2}(u,y)+a_{3}(x)d^{2}(x,v) \\
&&+k_{1}(x)d^{2}(u,x)+k_{2}(x)d^{2}(v,y)
\end{eqnarray*}
for all $x,y\in X,$there are $u\in Tx$ and $y\in Ty$ where $%
a_{1,}a_{2},a_{3},k_{1},k_{2}:X\rightarrow \lbrack 0,1]$ with $%
a_{1}(x)+a_{2}(x)+a_{3}(x)<1$, $2k_{1}(x)<1-a_{2}(x)$ and $%
2k_{2}(x)<1-a_{3}(x)$ for all $x\in X.$
\end{definition}

\begin{definition}
$T$ is called $(a_{1},a_{2},b_{1},b_{2})-$multivalued hybrid mapping type
II\ from $X$ to $CB(X)$ if for all $x,y\in X$%
\begin{equation*}
a_{1}(x)H^{2}(Tx,Ty)+a_{2}(x)d^{2}(Tx,y)\leq
b_{1}(x)d^{2}(x,Ty)+b_{2}(x)d^{2}(x,y)
\end{equation*}%
where $a_{1},a_{2}:X\rightarrow 
\mathbb{R}
$ and $b_{1},b_{2}:X\rightarrow 
\mathbb{R}
$ with $a_{1}(x)+a_{2}(x)\geq 1$ and $b_{1}(x)+b_{2}(x)\leq 1.$
\end{definition}

\begin{definition}
$T$ is called generalized multivalued hybrid mapping type II\ from $X$ to $%
CB(X)$ if for all $x,y\in X$%
\begin{eqnarray*}
H^{2}(Tx,Ty) &\leq
&a_{1}(x)d^{2}(x,y)+a_{2}(x)d^{2}(Tx,y)+a_{3}(x)d^{2}(x,Ty) \\
&&+k_{1}(x)d^{2}(Tx,x)+k_{2}(x)d^{2}(Ty,y)
\end{eqnarray*}%
where $a_{1,}a_{2},a_{3},k_{1},k_{2}:X\rightarrow \lbrack 0,1]$ with $%
a_{1}(x)+a_{2}(x)+a_{3}(x)<1$, $2k_{1}(x)<1-a_{2}(x)$ and $%
2k_{2}(x)<1-a_{3}(x)$ for all $x\in X.$
\end{definition}

\section{Existence Results}

\begin{proposition}
Let $X$ be a complete $CAT(\kappa )$ space, $K$ be a nonempty, closed and
convex subset of $X$ with $rad(K)<\frac{\pi }{2\sqrt{\kappa }}\ $and $T$ be $%
(a_{1},a_{2},b_{1},b_{2})-$ multivalued hybrid mapping type I\ from $K$ to $%
C(K)$ with $F(T)\neq \emptyset $ then $F(T)$ closed and $Tp=\{p\}$ for all $%
p\in F(T)$
\end{proposition}

\begin{proof}
Let $\{x_{n}\}$ be a sequence in $F(T)$ and $x_{n}\rightarrow x\in X.$ Then\
for any $u\in Tx,$we have

\begin{eqnarray*}
d^{2}(u,x_{n}) &\leq &a_{1}(x)d^{2}(u,x_{n})+a_{2}(x)d^{2}(u,x_{n}) \\
&\leq &b_{1}(x)d^{2}(x,x_{n})+b_{2}(x)d^{2}(x,x_{n}) \\
&\leq &d^{2}(x,x_{n})
\end{eqnarray*}

then taking limit on $n$ we have

\begin{equation*}
d(u,x)\leq 0
\end{equation*}

so $u=x\in Tx=\{x\}$
\end{proof}

\begin{proposition}
Let $\kappa >0$ and $X$ be a complete $CAT(\kappa )$ space and $K$ be a
nonempty closed convex subset of $X$ with $rad(K)\leq \frac{\pi -\varepsilon 
}{2\sqrt{\kappa }}$ for some $\varepsilon \in (0,\pi /2)$ and $%
T:K\rightarrow C(X)$ be a generalized multivalued hybrid mapping type I with 
$F(T)\neq \emptyset $ then $F(T)$ closed and $Tp=\{p\}$ for all $p\in F(T)$
\end{proposition}

\begin{proof}
Let $\{x_{n}\}$ be a sequence in $F(T)$ and $x_{n}\rightarrow x\in X.$ Then\
for any $u\in Tx,$we have

\begin{eqnarray*}
d^{2}(u,x_{n}) &\leq
&a_{1}(x)d^{2}(x,x_{n})+a_{2}(x)d^{2}(u,x_{n})+a_{3}(x)d^{2}(x,x_{n}) \\
&&+k_{1}(x)d^{2}(u,x)+k_{2}(x)d^{2}(x_{n},x_{n})
\end{eqnarray*}

implies that

\begin{equation*}
d(u,x_{n})\leq d^{2}(x,x_{n})+\frac{k_{1}(x)}{1-a_{2}(x)}d^{2}(u,x)
\end{equation*}%
then taking limit on $n$ we have

\begin{equation*}
(1-\frac{k_{1}(x)}{1-a_{2}(x)})d(u,x)\leq 0
\end{equation*}

so $u=x\in Tx=\{x\}$
\end{proof}

\begin{proposition}
Let $X$ be a complete $CAT(\kappa )$ space, $K$ be a nonempty, closed and
convex subset of $X$ with $rad(K)<\frac{\pi }{2\sqrt{\kappa }}\ $and $T$ be $%
(a_{1},a_{2},b_{1},b_{2})-$ multivalued hybrid mapping type II\ from $K$ to $%
C(K)$ with $F(T)\neq \emptyset $ and $a_{1}(p)\geq 1$ for all $p\in F(T)$
then $F(T)$ closed and $Tp=\{p\}$ for all $p\in F(T)$
\end{proposition}

\begin{proof}
Let $\{x_{n}\}$ be a sequence in $F(T)$ and $x_{n}\rightarrow x\in X.$ Then\
for any $Tx,$we have

\begin{eqnarray*}
d^{2}(Tx,x_{n}) &\leq &a_{1}(x)d^{2}(Tx,x_{n})+a_{2}(x)d^{2}(Tx,x_{n}) \\
&\leq &a_{1}(x)H^{2}(Tx,Tx_{n})+a_{2}(x)d^{2}(Tx,x_{n}) \\
&\leq &b_{1}(x)d^{2}(x,Tx_{n})+b_{2}(x)d^{2}(x,x_{n}) \\
&\leq &d^{2}(x,x_{n})
\end{eqnarray*}

then taking limit on $n$ we have

\begin{equation*}
d(Tx,x)\leq 0
\end{equation*}

so $u=x\in Tx=\{x\}$
\end{proof}

\begin{proposition}
Let $\kappa >0$ and $X$ be a complete $CAT(\kappa )$ space and $K$ be a
nonempty closed convex subset of $X$ with $rad(K)\leq \frac{\pi -\varepsilon 
}{2\sqrt{\kappa }}$ for some $\varepsilon \in (0,\pi /2)$ and $%
T:K\rightarrow KC(X)$ be a generalized multivalued hybrid mapping type II
with $F(T)\neq \emptyset $ then $F(T)$ closed and $Tp=\{p\}$ for all $p\in
F(T)$
\end{proposition}

\begin{proof}
Let $\{x_{n}\}$ be a sequence in $F(T)$ and $x_{n}\rightarrow x\in X.$ Then\
for any $u\in Tx,$we have we can find $y_{n}\in Tx_{n}$scuh that $%
d(u,y_{n})=d(u,Tx_{n}).d(x_{n},y_{n})\leq d\lim_{n\rightarrow \infty }d$

\begin{eqnarray*}
d^{2}(u,Tx_{n-1}) &\leq &H^{2}(Tx,Tx_{n-1}) \\
&\leq
&a_{1}(x)d^{2}(x,x_{n-1})+a_{2}(x)d^{2}(Tx,x_{n-1})+a_{3}(x)d^{2}(x,Tx_{n-1})
\\
&&+k_{1}(x)d^{2}(Tx,x)+k_{2}(x)d^{2}(Tx_{n-1},x_{n}) \\
&\leq
&a_{1}(x)d^{2}(x,x_{n-1})+a_{2}(x)d^{2}(u,x_{n-1})+a_{3}(x)d^{2}(x,x_{n}) \\
&&+k_{1}(x)d^{2}(u,x)
\end{eqnarray*}

implies that

\begin{equation*}
d^{2}(u,x_{n})\leq d^{2}(x,x_{n})+\frac{k_{1}(x)}{1-a_{2}(x)}d^{2}u,x)
\end{equation*}%
then taking limit on $n$ we have

\begin{equation*}
(1-\frac{k_{1}(x)}{1-a_{2}(x)})d(u,x)\leq 0
\end{equation*}

so $u=x\in Tx=\{x\}$
\end{proof}

\begin{theorem}
\label{(teo3.1)}Let $X$ be a complete $CAT(\kappa )$ space, $K$ be a
nonempty, closed and convex subset of $X$ with $rad(K)<\frac{\pi }{2\sqrt{%
\kappa }}\ $and $T$ be $(a_{1},a_{2},b_{1},b_{2})-$ multivalued hybrid
mapping type I\ from $K$ to $C(K)$ then $F(T)\neq \emptyset $ .
\end{theorem}

\begin{proof}
Let $x_{0}\in K$ and $x_{n}\in Tx_{n-1}$ for all $n\in 
\mathbb{N}
.$Assume that $A_{C}\{x_{n}\}=\{z\}$ then $z\in K$ by Lemma \ref{(lemma2.10)}%
. Since $T$ is type I, for all $n\in 
\mathbb{N}
$ and for any $u_{\text{ }}\in Tz$ such that 
\begin{equation*}
a_{1}(z)d^{2}(u,x_{n})+a_{2}(z)d^{2}(u,x_{n-1})\leq
b_{1}(z)d^{2}(z,x_{n})+b_{2}(z)d^{2}(z,x_{n-1})
\end{equation*}%
and taking limit superior on both side which implies that 
\begin{equation*}
\limsup_{n\rightarrow \infty }d^{2}(u,x_{n})\leq \limsup_{n\rightarrow
\infty }d^{2}(z,x_{n}).
\end{equation*}%
hencet $z=u\in Tz=\{u\}.$
\end{proof}

\begin{corollary}
Let $X$ be a complete $CAT(\kappa )$ space, $K$ be a nonempty closed convex
subset of $X$ with $rad(K)<\frac{\pi }{2\sqrt{\kappa }}\ $and $T$ be $%
(a_{1},a_{2},b_{1},b_{2})-$ hybrid mapping \ from $K$ to $K$ then $F(T)\neq
\emptyset $ .

\begin{proof}
Let take $F=\{T(x)\},$ then $F$ is $(a_{1},a_{2},b_{1},b_{2})-$multivalued
hybrid mapping \ from $K$ to $C(K)$.Hence $F$ has at least one fixed point
and so $T$ by Theorem \ref{(teo3.1)}
\end{proof}
\end{corollary}

\qquad Since for any $\kappa >\kappa ^{\prime },$ $CAT(\kappa ^{\prime })$
space is $CAT(\kappa )$ then following corollaries holds.

\begin{corollary}
Let $X$ be a complete $CAT(0)$ space, $K$ be a nonempty closed convex subset
of $X$ and $T$ be $(a_{1},a_{2},b_{1},b_{2})-$ multivalued hybrid mapping
type I\ from $K$ to $C(K).$ then there is a $x_{0}\in K$ such that the
sequence $\{x_{n}\}$ defined by $x_{n}\in Tx_{n-1}$ for all $n\in 
\mathbb{N}
$ is bounded if and only if $F(T)\neq \emptyset $.
\end{corollary}

\begin{corollary}
Let $X$ be a complete $CAT(0)$ space, $K$ be a nonempty closed convex subset
of $X$ and $T$ be $(a_{1},a_{2},b_{1},b_{2})-$hybrid mapping \ from $K$ to $%
K.$ then there is a $x_{0}\in K$ such that the sequence $\{x_{n}\}$ defined
by $x_{n}=Tx_{n-1}$ for all $n\in 
\mathbb{N}
$ is bounded if and only if $F(T)\neq \emptyset $.
\end{corollary}

\begin{theorem}
\label{(teo3.2)}Let $\kappa >0$ and $X$ be a complete $CAT(\kappa )$ space
and $K$ be a nonempty closed convex subset of $X$ with $rad(K)\leq \frac{\pi
-\varepsilon }{2\sqrt{\kappa }}$ for some $\varepsilon \in (0,\pi /2)$ and $%
T:K\rightarrow C(K)$ be a generalized multivalued hybrid mapping type I with 
$k_{1}(x)=k_{2}(x)=0$ for all $x\in K$ then $F(T)\neq \emptyset $.
\end{theorem}

\begin{proof}
Let $x_{0}\in K$ and $x_{n}\in Tx_{n-1}$ for all $n\in 
\mathbb{N}
.$Assume that $A_{C}\{x_{n}\}=\{z\}$ then $z\in K$ by Lemma \ref{(lemma2.10)}%
. Since $T$ is type I, for all $n\in 
\mathbb{N}
$ we can \ find a $u_{\text{ }}\in Tz$ such that multivalued hybrid mapping 
\begin{equation*}
d^{2}(x_{n},u)\leq
a_{1}(z)d^{2}(x_{n-1},z)+a_{2}(z)d^{2}(x_{n-1},u)+a_{3}(z)d^{2}(x_{n},z)
\end{equation*}%
and by taking limit superior on both side we get 
\begin{eqnarray*}
\limsup_{n\rightarrow \infty }d^{2}(x_{n},u) &\leq
&(a_{1}(z)+a_{3}(z))\limsup_{n\rightarrow \infty }d^{2}(x_{n},z) \\
&&+a_{2}(z)\limsup_{n\rightarrow \infty }d^{2}(x_{n-1},u)
\end{eqnarray*}%
implies that $\limsup_{n\rightarrow \infty }d(x_{n},u)\leq
\limsup_{n\rightarrow \infty }d(x_{n},z)$ which implies that $z=u\in Tz.$
\end{proof}

\begin{corollary}
Let $\kappa >0$ and $X$ be a complete $CAT(\kappa )$ space, $K$ be a
nonempty closed convex subset of $X$ with $rad(K)\leq \frac{\pi -\varepsilon 
}{2\sqrt{\kappa }}$ for some $\varepsilon \in (0,\pi /2)$ and $%
T:K\rightarrow K$ be a generalized hybrid mapping type I with $%
k_{1}(x)=k_{2}(x)=0$ for all $x\in K$ then $F(T)\neq \emptyset $.
\end{corollary}

\begin{proof}
Let take $F=\{T(x)\},$ then $F$ is multivalued hybrid mapping \ from $K$ to $%
C(K)$.Hence $F$ has at least one fixed point and so $T$ by Theorem \ref%
{(teo3.2)}.
\end{proof}

\begin{corollary}
Let $X$ be a complete $CAT(0)$ space and $K$ be a nonempty closed convex
subset of $X$, and $T:K\rightarrow K(K)$ be a generalized multivalued hybrid
mapping type I with $k_{1}(x)=k_{2}(x)=0$ for all $x\in K$ then there is a $%
x_{0}\in K$ such that the sequence $\{x_{n}\}$ defined by $x_{n}\in Tx_{n-1}$
for all $n\in 
\mathbb{N}
$ is bounded if and only if $F(T)\neq \emptyset $
\end{corollary}

\begin{corollary}
Let $X$ be a complete $CAT(0)$ space and $K$ be a nonempty closed convex
subset of $X$, and $T:K\rightarrow K$ be a generalized hybrid mapping with $%
k_{1}(x)=k_{2}(x)=0$ for all $x\in K.$Then there is a $x_{0}\in K$ such that
the sequence $\{x_{n}\}$ defined by $x_{n}=Tx_{n-1}$ for all $n\in 
\mathbb{N}
$ is bounded if and only if $F(T)\neq \emptyset $.
\end{corollary}

\begin{theorem}
Let $\kappa >0$ and $X$ be a complete $CAT(\kappa )$ space with $diam(X)\leq 
\frac{\pi -\varepsilon }{2\sqrt{\kappa }}$ for some $\varepsilon \in (0,\pi
/2)$. Let $K$ be a nonempty closed convex subset of $X$, and $T:K\rightarrow
K(K)$ be a generalized multivalued hybrid mapping type II satisfying either

\begin{enumerate}
\item[i)] $a_{2}(x)=0$ and $\frac{2k_{2}(x)}{1-a_{3}(x)}<\frac{R}{2}$ or

\item[ii)] $a_{3}(x)=0$ and $\frac{2k_{1}(x)}{1-a_{2}(x)}<\frac{R}{2}$
\end{enumerate}

for all $x\in K$ where $R=(\pi -2\varepsilon )tan(\varepsilon )$.Moreover,
If $k=\sup \frac{a_{1}(x)+k_{1}(x)}{1-k_{2}(x)}<1$\ then $F(T)\neq \emptyset 
$.
\end{theorem}

\begin{proof}
Let $x_{0}\in K$ and $x_{n+1}\in Tx_{n}$ such that $%
d(x_{n+1},x_{n})=d(Tx_{n},x_{n})$ for all $n\in 
\mathbb{N}
.$ Assume that $a_{2}(x)=0.$Then%
\begin{eqnarray*}
d^{2}(x_{n+1},x_{n}) &=&d^{2}(Tx_{n},x_{n})\leq H^{2}(Tx_{n},Tx_{n-1}) \\
&\leq &a_{1}(x_{n})d^{2}(x_{n},x_{n-1})+a_{3}(x)d^{2}(Tx_{n-1},x_{n}) \\
&&+k_{1}(x_{n})d^{2}(Tx_{n},x_{n})+k_{2}(x_{n})d^{2}(Tx_{n-1},x_{n-1}) \\
&\leq &a_{1}(x_{n})d^{2}(x_{n},x_{n-1})+k_{1}(x_{n})d^{2}(Tx_{n},x_{n}) \\
&&+k_{2}(x_{n})d^{2}(Tx_{n-1},x_{n-1})
\end{eqnarray*}%
implies that%
\begin{equation*}
d^{2}(x_{n+1},x_{n})\leq \frac{a_{1}(x_{n})+k_{1}(x_{n})}{1-k_{2}(x_{n})}%
d^{2}(x_{n},x_{n-1})
\end{equation*}%
hence we have%
\begin{eqnarray*}
d(x_{n+1},x_{n}) &\leq &\sqrt{\frac{a_{1}(x_{n})+k_{1}(x_{n})}{1-k_{2}(x_{n})%
}}d(x_{n},x_{n-1}) \\
&\leq &kd(x_{n},x_{n-1}) \\
&\leq &k^{n}d(x_{1},x_{0}).
\end{eqnarray*}%
Let $n<m,$then%
\begin{eqnarray*}
d(x_{m},x_{n}) &\leq &\sum\nolimits_{i=n}^{i=m-1}d(x_{i+1},x_{i}) \\
&\leq &\sum\nolimits_{i=n}^{i=m-1}k^{i+1}d(x_{1},x_{0}) \\
&\leq &d(x_{1},x_{0})\sum\nolimits_{i=n}^{i=m-1}k^{i+1}.
\end{eqnarray*}%
Since $k<1$ then the sequence $(x_{n})$ is Cauchy sequence and since space
is complete than $x_{n}\rightarrow z\in X.$%
\begin{eqnarray*}
d^{2}(x_{n},Tz) &\leq &H^{2}(Tx_{n-1},Tz) \\
&\leq &a_{1}(z)d^{2}(x_{n-1},z)+a_{3}(z)d^{2}(x_{n-1},Tz) \\
&&+k_{1}(z)d^{2}(Tx_{n-1},x_{n-1})+k_{2}(z)d^{2}(Tz,z)
\end{eqnarray*}%
on the other hand%
\begin{equation*}
d^{2}(x_{n},\frac{1}{2}Tz\oplus \frac{1}{2}z)\leq \frac{1}{2}d^{2}(x_{n},Tz)+%
\frac{1}{2}d^{2}(x_{n},z)-\frac{R}{8}d^{2}(z,Tz)
\end{equation*}%
which implies that%
\begin{equation*}
d^{2}(z,Tz)\leq \frac{4}{R}d^{2}(x_{n},Tz)+\frac{4}{R}d^{2}(x_{n},z)
\end{equation*}%
so we we have%
\begin{eqnarray*}
d^{2}(x_{n},Tz) &\leq &a_{1}(z)d^{2}(x_{n-1},z)+a_{3}(z)d^{2}(x_{n-1},Tz) \\
&&+k_{1}(z)d^{2}(Tx_{n-1},x_{n-1})+k_{2}(z)d^{2}(Tz,z) \\
&\leq
&a_{1}(z)d^{2}(x_{n-1},z)+a_{3}(z)d^{2}(x_{n-1},Tz)+k_{1}(z)d^{2}(Tx_{n-1},x_{n-1})
\\
&&+k_{2}(z)(\frac{4}{R}d^{2}(x_{n},Tz)+\frac{4}{R}d^{2}(x_{n},z))
\end{eqnarray*}%
which implies that%
\begin{equation*}
(1-k_{2}(z)\frac{4}{R}-a_{3}(z))\lim_{n\rightarrow \infty
}d^{2}(x_{n},Tz)\leq 0
\end{equation*}%
so $\lim_{n\rightarrow \infty }d^{2}(x_{n},Tz)=0.$ Hence $d^{2}(z,Tz)\leq
d^{2}(x_{n},z)+d^{2}(x_{n},Tz)\rightarrow \infty $ implies that $z\in Tz.$
\end{proof}

\section{Convergence Results}

\begin{theorem}
\label{(teo4.1)}(Demiclosed principle for $(a_{1},a_{2},b_{1},b_{2})-$%
multivalued hybrid mapping type I )Let $\kappa >0$ and $X$ be a complete $%
CAT(\kappa )$ space with $diam(X)<\frac{\pi }{2\sqrt{\kappa }}$. Let $K$ be
a nonempty closed convex subset of $X$, and $T:K\rightarrow C(X)$ be a $%
(a_{1},a_{2},b_{1},b_{2})-$multivalued hybrid mapping type I. Let$\{x_{n}\}$
be a sequence in $K$ with $\Delta -\lim_{n\rightarrow \infty }x_{n}=z$ and $%
\lim_{n\rightarrow \infty }d(x_{n},Tx_{n})=0$. Then $z\in K$ and $z\in T(z).$
\end{theorem}

\begin{proof}
By Lemma \ref{(lemma2.10)}, $z\in K.$We can find a sequence $\{y_{n}\}$ such
that $d(x_{n},y_{n})=d(x_{n},Tx_{n}),$ so we have $\lim_{n\rightarrow \infty
}d(x_{n},y_{n})=0.$Because of $T$ is $(a_{1},a_{2},b_{1},b_{2})-$multivalued
hybrid mapping type I, for all $u\in Tz$ such that 
\begin{equation*}
a_{1}(z)d^{2}(u,y_{n})+a_{2}(z)d^{2}(u,x_{n})\leq
b_{1}(z)d^{2}(z,y_{n})+b_{2}(z)d^{2}(z,x_{n})
\end{equation*}

Then by triangular inequality we have $d(x_{n},u)\leq
d(x_{n},y_{n})+d(y_{n},u).$ So we have, $\limsup_{n\rightarrow \infty
}d(x_{n},u)\leq \limsup_{n\rightarrow \infty }d(y_{n},u)$ and again since $%
d(y_{n},u)\leq d(y_{n},x_{n})+d(x_{n},u)$ we have $\limsup_{n\rightarrow
\infty }d(y_{n},u)\leq \limsup_{n\rightarrow \infty }d(x_{n},u),$ combining
these we have that $\limsup_{n\rightarrow \infty
}d(x_{n},u)=\limsup_{n\rightarrow \infty }d(y_{n},u)$. So we have that 
\begin{eqnarray*}
a_{1}(z)d^{2}(u,y_{n})+a_{2}(z)d^{2}(u,x_{n}) &\leq
&b_{1}(z)d^{2}(z,y_{n})+b_{2}(z)d^{2}(z,x_{n}) \\
&\leq &b_{1}(z)[d(z,x_{n})+d(x_{n},y_{n})]^{2}+b_{2}(z)d^{2}(x_{n},z)
\end{eqnarray*}

which implies that $\limsup_{n\rightarrow \infty }d(u,x_{n})\leq
\limsup_{n\rightarrow \infty }d(z,x_{n})$ Then $z=u\in Tz.$Assume that $%
a_{2}(z)>0.$
\end{proof}

\begin{corollary}
Let $\kappa >0$ and $X$ be a complete $CAT(\kappa )$ space with $diam(X)<%
\frac{\pi }{2\sqrt{\kappa }}$. Let $K$ be a nonempty closed convex subset of 
$X$, and $T:K\rightarrow X$ be a $(a_{1},a_{2},b_{1},b_{2})-$ hybrid mapping
. Let$\{x_{n}\}$ be a sequence in $K$ with $\Delta -\lim_{n\rightarrow
\infty }x_{n}=z$ and $\lim_{n\rightarrow \infty }d(x_{n},Tx_{n})=0$. Then $%
z\in K$ and $Tz=z.$
\end{corollary}

\begin{corollary}
Let $X$ be a complete $CAT(0)$ space and $K$ be a nonempty closed convex
subset of $X$, and $T:K\rightarrow C(X)$ be a $(a_{1},a_{2},b_{1},b_{2})-$%
multivalued hybrid mapping type I. Let$\{x_{n}\}$ be a bounded sequence in $%
K $ with $\Delta -\lim_{n\rightarrow \infty }x_{n}=z$ and $%
\lim_{n\rightarrow \infty }d(x_{n},Tx_{n})=0$. Then $z\in K$ and $z\in Tz.$
\end{corollary}

\begin{corollary}
Let $X$ be a complete $CAT(0)$ space and $K$ be a nonempty closed convex
subset of $X$, and $T:K\rightarrow X$ be a $(a_{1},a_{2},b_{1},b_{2})-$
hybrid mapping. Let $\{x_{n}\}$ be a bounded sequence in $K$ with $\Delta
-\lim_{n\rightarrow \infty }x_{n}=z$ and $\lim_{n\rightarrow \infty
}d(x_{n},Tx_{n})=0$. Then $z\in K$ and $Tz=z.$
\end{corollary}

\begin{theorem}
\label{(teo4.2)}(Demiclosed principle for generalized multivalued hybrid
mapping type I)Let $\kappa >0$ and $X$ be a complete $CAT(\kappa )$ space
with $diam(X)\leq \frac{\pi -\varepsilon }{2\sqrt{\kappa }}$ for some $%
\varepsilon \in (0,\pi /2)$. Let $K$ be a nonempty closed convex subset of $%
X $, and $T:K\rightarrow C(X)$ be a generalized multivalued hybrid mapping
type I with $\frac{2k_{1}(x)}{1-a_{2}(x)}<\frac{R}{2}$ for all $x\in K$
where $R=(\pi -2\varepsilon )tan(\varepsilon )$. Let$\{x_{n}\}$ be a
sequence in $K$ with $\Delta -\lim_{n\rightarrow \infty }x_{n}=z$ and $%
\lim_{n\rightarrow \infty }d(x_{n},Tx_{n})=0$. Then $z\in K$ and $z\in T(z).$
\end{theorem}

\begin{proof}
By Lemma \ref{(lemma2.10)}, $z\in K.$We can find a sequence $\{y_{n}\}$ such
that $d(x_{n},y_{n})=d(x_{n},Tx_{n}),$ so we have $\lim_{n\rightarrow \infty
}d(x_{n},y_{n})=0.$Since $T$ is generalized multivalued hybrid mapping type
I, for all $u\in Tz$ such that,%
\begin{eqnarray*}
d^{2}(y_{n},u) &\leq
&a_{1}(z)d^{2}(x_{n},z)+a_{2}(z)d^{2}(u,x_{n})+a_{3}(z)d^{2}(y_{n},z) \\
&&+k_{1}(z)d^{2}(y_{n},x_{n})+k_{2}(x)d^{2}(u,z) \\
&\leq &a_{1}(z)d^{2}(x_{n},z)+a_{2}(z)[d(x_{n},y_{n})+d(y_{n},u)]^{2} \\
&&+a_{3}(z)[d(y_{n},x_{n})+d(x_{n},z)]^{2} \\
&&+k_{1}(z)d^{2}(y_{n},x_{n})+k_{2}(x)d^{2}(u,z)
\end{eqnarray*}%
which implies that 
\begin{equation*}
\limsup_{n\rightarrow \infty }d^{2}(y_{n},u)\leq \limsup_{n\rightarrow
\infty }d^{2}(x_{n},z)+\frac{k_{2}(x)}{1-a_{2}(x)}d^{2}(z,u)
\end{equation*}%
and using this we have 
\begin{eqnarray*}
\limsup_{n\rightarrow \infty }d^{2}(x_{n},u) &\leq &\limsup_{n\rightarrow
\infty }[d^{2}(x_{n},y_{n})+d^{2}(y_{n},u)] \\
&\leq &\limsup_{n\rightarrow \infty }d^{2}(x_{n},z)+\frac{k_{1}(x)}{%
1-a_{2}(x)}d^{2}(z,u).
\end{eqnarray*}%
By CN$^{x}$ inequality we have%
\begin{equation*}
d^{2}(x_{n},\frac{1}{2}z\oplus \frac{1}{2}u)\leq \frac{1}{2}d^{2}(x_{n},z)+%
\frac{1}{2}d^{2}(x_{n},u)-\frac{R}{8}d^{2}(z,u)
\end{equation*}%
and combining all of these we get%
\begin{eqnarray*}
\limsup_{n\rightarrow \infty }d^{2}(x_{n},\frac{1}{2}z\oplus \frac{1}{2}u)
&\leq &\frac{1}{2}\limsup_{n\rightarrow \infty }d^{2}(x_{n},z)+\frac{1}{2}%
\limsup_{n\rightarrow \infty }d^{2}(x_{n},u) \\
&&-R\frac{1}{8}d^{2}(z,u) \\
&\leq &\frac{1}{2}\limsup_{n\rightarrow \infty }d^{2}(x_{n},z)+\frac{1}{2}%
\limsup_{n\rightarrow \infty }d^{2}(x_{n},z) \\
&&+\frac{k_{1}(x)}{2(1-a_{2}(x))}d^{2}(z,u)-\frac{R}{8}d^{2}(z,u). \\
&=&\frac{1}{2}\limsup_{n\rightarrow \infty }d^{2}(x_{n},z)+\frac{1}{2}%
\limsup_{n\rightarrow \infty }d^{2}(x_{n},z) \\
&&+(\frac{k_{1}(x)}{2(1-a_{2}(x))}-\frac{R}{8})d^{2}(z,u) \\
&=&\frac{1}{2}\limsup_{n\rightarrow \infty }d^{2}(x_{n},z)+\frac{1}{2}%
\limsup_{n\rightarrow \infty }d^{2}(x_{n},z) \\
&&+(\frac{k_{1}(x)}{2(1-a_{2}(x))}-\frac{R}{8})d^{2}(z,u)
\end{eqnarray*}%
which implies that 
\begin{equation*}
(\frac{R}{8}-\frac{k_{1}(x)}{2(1-a_{2}(x))})d^{2}(z,u)\leq
\limsup_{n\rightarrow \infty }d^{2}(x_{n},z)-\limsup_{n\rightarrow \infty
}d^{2}(x_{n},\frac{1}{2}z\oplus \frac{1}{2}u)\leq 0
\end{equation*}%
and by assumptions we have $z=u\in Tz$
\end{proof}

\begin{corollary}
Let $\kappa >0$ and $X$ be a complete $CAT(\kappa )$ space with $diam(X)\leq 
\frac{\pi -\varepsilon }{2\sqrt{\kappa }}$ for some $\varepsilon \in (0,\pi
/2)$. Let $K$ be a nonempty closed convex subset of $X$, and $T:K\rightarrow
X$ be a generalized hybrid mapping with $\frac{2k_{1}(x)}{1-a_{2}(x)}<R$ for
all $x\in K$ where $R=(\pi -2\varepsilon )tan(\varepsilon )$. Let$\{x_{n}\}$
be a sequence in $K$ with $\Delta -\lim_{n\rightarrow \infty }x_{n}=z$ and $%
\lim_{n\rightarrow \infty }d(x_{n},Tx_{n})=0$. Then $z\in K$ and $T(z)=z.$
\end{corollary}

\begin{corollary}
Let $X$ be a complete $CAT(0)$ space and $K$ be a nonempty closed convex
subset of $X$, and $T:K\rightarrow C(X)$ be a generalized multivalued hybrid
mapping type I . Let$\{x_{n}\}$ be a sequence in $K$ with $\Delta
-\lim_{n\rightarrow \infty }x_{n}=z$ and $\lim_{n\rightarrow \infty
}d(x_{n},Tx_{n})=0$. Then $z\in K$ and $z\in T(z).$
\end{corollary}

\begin{corollary}
Let $X$ be a complete $CAT(0)$ space and $K$ be a nonempty closed convex
subset of $X$, and $T:K\rightarrow X$ be a generalized hybrid mapping. Let $%
\{x_{n}\}$ be a sequence in $K$ with $\Delta -\lim_{n\rightarrow \infty
}x_{n}=z$ and $\lim_{n\rightarrow \infty }d(x_{n},Tx_{n})=0$. Then $z\in K$
and $Tz=z.$
\end{corollary}

\begin{theorem}
\label{(teo4.3)}(Demiclosed principle for generalized multivalued hybrid
mapping type II)Let $\kappa >0$ and $X$ be a complete $CAT(\kappa )$ space
with $diam(X)\leq \frac{\pi -\varepsilon }{2\sqrt{\kappa }}$ for some $%
\varepsilon \in (0,\pi /2)$. Let $K$ be a nonempty closed convex subset of $%
X $, and $T:K\rightarrow KC(X)$ be a generalized multivalued hybrid mapping
with $\frac{2k_{1}(x)}{1-a_{2}(x)}<\frac{R}{2}$ for all $x\in K$ where $%
R=(\pi -2\varepsilon )tan(\varepsilon )$. Let$\{x_{n}\}$ be a sequence in $K$
with $\Delta -\lim_{n\rightarrow \infty }x_{n}=z$ and $\lim_{n\rightarrow
\infty }d(x_{n},Tx_{n})=0$. Then $z\in K$ and $z\in T(z).$
\end{theorem}

\begin{proof}
By Lemma 2, $z\in K.$We can find a sequence $\{y_{n}\}$ such that $y_{n}\in
Tx_{n},$ $d(x_{n},y_{n})=d(x_{n},Tx_{n}),$ so we have $\lim_{n\rightarrow
\infty }d(x_{n},y_{n})=0$ and since $Tz$ is compact we can find a sequence $%
\{z_{n}\}$ in $Tz$ such that $d(y_{n},z_{n})=d(y_{n},Tz).$ Then there is a
convergent subsequence $\{z_{n_{i}}\}$ of $\{z_{n}\},$ say $%
\lim_{i\rightarrow \infty }z_{n}=u\in Tz$.%
\begin{eqnarray*}
d(x_{n_{i}},u) &\leq
&d(x_{n_{i}},y_{n_{i}})+d(y_{n_{i}},z_{n_{i}})+d(z_{n_{i}},u) \\
&\leq &d(x_{n_{i}},y_{n_{i}})+d(y_{n_{i}},Tz)+d(z_{n_{i}},u) \\
&\leq &d(x_{n_{i}},y_{n_{i}})+H(Tx_{n_{i}},Tz)+d(z_{n_{i}},u) \\
&\leq &d(x_{n_{i}},y_{n_{i}})+H(Tx_{n_{i}},Tz)+d(z_{n_{i}},u)
\end{eqnarray*}%
implies that \ $\limsup_{n\rightarrow \infty }d(x_{n_{i}},u)\leq
\limsup_{n\rightarrow \infty }H(Tx_{n},Tz)$.Because of $T$ is generalized
multivalued hybrid mapping,%
\begin{eqnarray*}
H^{2}(Tx_{n},Tz) &\leq
&a_{1}(z)d^{2}(x_{n},z)+a_{2}(z)d^{2}(Tz,x_{n})+a_{3}(z)d^{2}(Tx_{n},z) \\
&&+k_{1}(z)d^{2}(Tx_{n},x_{n})+k_{2}(x)d^{2}(Tz,z) \\
&\leq &a_{1}(z)d^{2}(x_{n},z)+a_{2}(z)[d(x_{n},Tx_{n})+H(Tx_{n},Tz)]^{2} \\
&&+a_{3}(z)[d(Tx_{n},x_{n})+d(x_{n},z)]^{2}+k_{1}(z)d^{2}(Tx_{n},x_{n})+k_{2}(x)d^{2}(Tz,z)
\end{eqnarray*}%
which implies that 
\begin{eqnarray*}
\limsup_{n\rightarrow \infty }H^{2}(Tx_{n},Tz) &\leq &\limsup_{n\rightarrow
\infty }d^{2}(x_{n},z)+\frac{k_{1}(x)}{1-a_{2}(x)}d^{2}(z,Tz) \\
&\leq &\limsup_{n\rightarrow \infty }d^{2}(x_{n},z)+\frac{k_{1}(x)}{%
1-a_{2}(x)}d^{2}(z,u).
\end{eqnarray*}%
By CN$^{x}$ inequality we have%
\begin{equation*}
d^{2}(x_{n},\frac{1}{2}z\oplus \frac{1}{2}u)\leq \frac{1}{2}d^{2}(x_{n},z)+%
\frac{1}{2}d^{2}(x_{n},u)-\frac{R}{8}d^{2}(z,u)
\end{equation*}%
and combining all of these we get%
\begin{eqnarray*}
\limsup_{n\rightarrow \infty }d^{2}(x_{n},\frac{1}{2}z\oplus \frac{1}{2}u)
&\leq &\frac{1}{2}\limsup_{n\rightarrow \infty }d^{2}(x_{n},z)+\frac{1}{2}%
\limsup_{n\rightarrow \infty }d^{2}(x_{n},u)-\frac{R}{8}d^{2}(z,u) \\
&\leq &\frac{1}{2}\limsup_{n\rightarrow \infty }d^{2}(x_{n},z)+\frac{1}{2}%
\limsup_{n\rightarrow \infty }H(Tx_{n},Tz)-\frac{R}{8}d^{2}(z,u). \\
&\leq &\frac{1}{2}\limsup_{n\rightarrow \infty }d^{2}(x_{n},z)+\frac{1}{2}%
\limsup_{n\rightarrow \infty }d^{2}(x_{n},z) \\
&&+\frac{k_{1}(x)}{2(1-a_{2}(x))}d^{2}(z,u)-\frac{R}{8}d^{2}(z,u). \\
&=&\limsup_{n\rightarrow \infty }d^{2}(x_{n},z)+(\frac{k_{1}(x)}{%
2(1-a_{2}(x))}-\frac{R}{8})d^{2}(z,u)
\end{eqnarray*}%
which implies that 
\begin{equation*}
(\frac{R}{8}-\frac{k_{1}(x)}{2(1-a_{2}(x))})d^{2}(z,u)\leq
\limsup_{n\rightarrow \infty }d^{2}(x_{n},z)-\limsup_{n\rightarrow \infty
}d^{2}(x_{n},\frac{1}{2}z\oplus \frac{1}{2}u)\leq 0
\end{equation*}%
and by assumptions we have $z=u\in Tz.$
\end{proof}

\begin{corollary}
Let $X$ be a complete $CAT(0)$ space and $K$ be a nonempty closed convex
subset of $X$, and $T:K\rightarrow KC(X)$ be a generalized multivalued
hybrid mapping type II. Let$\{x_{n}\}$ be a sequence in $K$ with $\Delta
-\lim_{n\rightarrow \infty }x_{n}=z$ and $\lim_{n\rightarrow \infty
}d(x_{n},Tx_{n})=0$. Then $z\in K$ and $z\in T(z).$
\end{corollary}

\begin{theorem}
\label{(teo4.1) copy(1)}(Demiclosed principle for $(a_{1},a_{2},b_{1},b_{2})-
$multivalued hybrid mapping type II )Let $\kappa >0$ and $X$ be a complete $%
CAT(\kappa )$ space with $diam(X)\leq \frac{\pi -\varepsilon }{2\sqrt{\kappa 
}}$ for some $\varepsilon \in (0,\pi /2)$. Let $K$ be a nonempty convex
compact subset of $X$, and $T:K\rightarrow KC(X)$ be a $%
(a_{1},a_{2},b_{1},b_{2})-$multivalued hybrid mapping type II with $%
a_{1}(x)\geq 1$ for all $x\in K$. Let$\{x_{n}\}$ be a sequence in $K$ with $%
\Delta -\lim_{n\rightarrow \infty }x_{n}=z$ and $\lim_{n\rightarrow \infty
}d(x_{n},Tx_{n})=0$. Then $z\in K$ and $z\in T(z).$
\end{theorem}

\begin{proof}
By Lemma \ref{(lemma2.10)}, $z\in K..$We can find a sequence $\{y_{n}\}$
such that $y_{n}\in Tx_{n},$ $d(x_{n},y_{n})=d(x_{n},Tx_{n}),$ so we have $%
\lim_{n\rightarrow \infty }d(x_{n},y_{n})=0$ and since $Tz$ is compact we
can find a sequence $\{z_{n}\}$ in $Tz$ such that $%
d(y_{n},z_{n})=d(y_{n},Tz).$ Then there is a convergent subsequence $%
\{z_{n_{i}}\}$ of $\{z_{n}\},$ say $\lim_{i\rightarrow \infty
}z_{n_{i}}=u\in Tz$.

\begin{eqnarray*}
d(x_{n_{i}},u) &\leq
&d(x_{n_{i}},y_{n_{i}})+d(y_{n_{i}},z_{n_{i}})+d(z_{n_{i}},u) \\
&\leq &d(x_{n_{i}},y_{n_{i}})+d(y_{n_{i}},Tz)+d(z_{n_{i}},u) \\
&\leq &d(x_{n_{i}},y_{n_{i}})+H(Tx_{n_{i}},Tz)+d(z_{n_{i}},u) \\
&\leq &d(x_{n_{i}},y_{n_{i}})+H(Tx_{n_{i}},Tz)+d(z_{n_{i}},u)
\end{eqnarray*}

implies that \ $\limsup_{n\rightarrow \infty }d(x_{n_{i}},u)\leq
\limsup_{n\rightarrow \infty }H(Tx_{n_{i}},Tz).$Because of $T$ is $%
(a_{1},a_{2},b_{1},b_{2})-$multivalued hybrid mapping type II, 
\begin{equation*}
a_{1}(z)H^{2}(Tz,Tx_{n_{i}})\leq
b_{1}(z)d^{2}(z,x_{n_{i}})+b_{2}(z)d^{2}(z,x_{n_{i}})-a_{2}(z)d^{2}(z,Tx_{n_{i}})
\end{equation*}

Then by triangular inequality we have $d(x_{n},u)\leq
d(x_{n},y_{n})+d(y_{n},u).$ So we have, $\limsup_{n\rightarrow \infty
}d(x_{n},u)\leq \limsup_{n\rightarrow \infty }d(y_{n},u)$ and again since $%
d(y_{n},u)\leq d(y_{n},x_{n})+d(x_{n},u)$ we have $\limsup_{n\rightarrow
\infty }d(y_{n},u)\leq \limsup_{n\rightarrow \infty }d(x_{n},u),$ combining
these we have that $\limsup_{n\rightarrow \infty
}d(x_{n},u)=\limsup_{n\rightarrow \infty }d(y_{n},u)$. So we have that 
\begin{eqnarray*}
&\leq &a_{1}(z)d^{2}(y_{n},u)+a_{2}(z)d^{2}(x_{n},u)\leq
b_{1}(z)d^{2}(y_{n},z)+b_{2}(z)d^{2}(x_{n},z) \\
&\leq &b_{1}(z)[d(y_{n},x_{n})+d(x_{n},z)]^{2}+b_{2}(z)d^{2}(x_{n},z)
\end{eqnarray*}

which implies that $\limsup_{n\rightarrow \infty }d(x_{n},u)\leq
\limsup_{n\rightarrow \infty }d(x_{n},z)$ Then $z=u\in Tz.$Assume that $%
a_{2}(z)>0.$
\end{proof}

\begin{lemma}
\label{(lemma4.1)}Let $\kappa >0$ and $X$ be a complete $CAT(\kappa )$ space
with $diam(X)\leq \frac{\pi -\varepsilon }{2\sqrt{\kappa }}$ for some $%
\varepsilon \in (0,\pi /2)$. Let $K$ be a nonempty closed convex subset of $%
X $, and $T:K\rightarrow C(X)$ be a generalized multivalued hybrid mapping
type I with $\frac{2k_{1}(x)}{1-a_{2}(x)}<R$ for all $x\in K$ where $R=(\pi
-2\varepsilon )tan(\varepsilon )$. Let$\{x_{n}\}$ be a sequence in $K$ with $%
\lim_{n\rightarrow \infty }d(x_{n},Tx_{n})=0$ and $\{d(x_{n},p)\}$ converges
for all $p\in F(T)$. Then $\omega _{w}(x_{n})\subseteq F(T)$ and $\omega
_{w}(x_{n})$ include exactly one point.
\end{lemma}

\begin{proof}
Let take $u\in \omega _{w}(x_{n})$ then there exist subsequence $\{u_{n}\}$
of $\{x_{n}\}$ with $A(\{u_{n}\})=\{u\}.$Then By Lemma \ref{(lemma2.10)}
there exist subsequence $\{v_{n}\}$ of $\{u_{n}\}$ with $\Delta
-\lim_{n\rightarrow \infty }v_{n}=v\in K$ . Then by Theorem \ref{(teo4.1)}
we have $v\in F(T)$ and by Lemma \ref{(lemma2.11)} we conclude that $u=v,$
hence we get $\omega _{w}(x_{n})\subseteq F(T)$. Let take subsequence $%
\{u_{n}\}$ of $\{x_{n}\}$with $A(\{u_{n}\})=\{u\}$ and $A(\{x_{n}\})=\{x\}.$
Because of $v\in \omega _{w}(x_{n})\subseteq F(T)$, $\{d(x_{n},u)\}$
converges, so by Lemma \ref{(lemma2.11)} we have $x=u,$ this means that $%
\omega _{w}(x_{n})$ include exactly one point.
\end{proof}

\begin{lemma}
\label{(lemma4.2)}Let $\kappa >0$ and $X$ be a complete $CAT(\kappa )$ space
with $diam(X)\leq \frac{\pi -\varepsilon }{2\sqrt{\kappa }}$ for some $%
\varepsilon \in (0,\pi /2)$. Let $K$ be a nonempty closed convex subset of $%
X $, and $T:K\rightarrow C(X)$ be $(a_{1},a_{2},b_{1},b_{2})-$multivalued
hybrid mapping type I. Let$\{x_{n}\}$ be a sequence in $K$ with $%
\lim_{n\rightarrow \infty }d(x_{n},Tx_{n})=0$ and $\{d(x_{n},p)\}$ converges
for all $p\in F(T)$. Then $\omega _{w}(x_{n})\subseteq F(T)$ and $\omega
_{w}(x_{n})$ include exactly one point.
\end{lemma}

\begin{proof}
Let take $u\in \omega _{w}(x_{n})$ then there exist subsequence $\{u_{n}\}$
of $\{x_{n}\}$ with $A(\{u_{n}\})=\{u\}.$Then By Lemma \ref{(lemma2.10)}
there exist subsequence $\{v_{n}\}$ of $\{u_{n}\}$ with $\Delta
-\lim_{n\rightarrow \infty }v_{n}=v\in K$ . Then by Theorem \ref{(teo4.2)}
we have $v\in F(T)$ and by Lemma \ref{(lemma2.11)} we conclude that $u=v,$
hence we get $\omega _{w}(x_{n})\subseteq F(T)$. Let take subsequence $%
\{u_{n}\}$ of $\{x_{n}\}$with $A(\{u_{n}\})=\{u\}$ and $A(\{x_{n}\})=\{x\}.$
Because of $v\in \omega _{w}(x_{n})\subseteq F(T)$, $\{d(x_{n},u)\}$
converges, so by Lemma \ref{(lemma2.11)} we have $x=u,$ this means that $%
\omega _{w}(x_{n})$ include exactly one point.
\end{proof}

\begin{lemma}
\label{(lemma4.3)}Let $\kappa >0$ and $X$ be a complete $CAT(\kappa )$ space
with $diam(X)\leq \frac{\pi -\varepsilon }{2\sqrt{\kappa }}$ for some $%
\varepsilon \in (0,\pi /2)$. Let $K$ be a nonempty closed convex subset of $%
X $, and $T:K\rightarrow KC(X)$ be a generalized multivalued hybrid mapping
type II with $\frac{2k_{1}(x)}{1-a_{2}(x)}<R$ for all $x\in K$ where $R=(\pi
-2\varepsilon )tan(\varepsilon )$. Let$\{x_{n}\}$ be a sequence in $K$ with $%
\lim_{n\rightarrow \infty }d(x_{n},Tx_{n})=0$ and $\{d(x_{n},p)\}$ converges
for all $p\in F(T)$. Then $\omega _{w}(x_{n})\subseteq F(T)$ and $\omega
_{w}(x_{n})$ include exactly one point.
\end{lemma}

\begin{proof}
Let take $u\in \omega _{w}(x_{n})$ then there exist subsequence $\{u_{n}\}$
of $\{x_{n}\}$ with $A(\{u_{n}\})=\{u\}.$Then By Lemma \ref{(lemma2.10)}
there exist subsequence $\{v_{n}\}$ of $\{u_{n}\}$ with $\Delta
-\lim_{n\rightarrow \infty }v_{n}=v\in K$ . Then by Theorem \ref{(teo4.3)}
we have $v\in F(T)$ and by Lemma \ref{(lemma2.11)} we conclude that $u=v,$
hence we get $\omega _{w}(x_{n})\subseteq F(T)$. Let take subsequence $%
\{u_{n}\}$ of $\{x_{n}\}$with $A(\{u_{n}\})=\{u\}$ and $A(\{x_{n}\})=\{x\}.$
Because of $v\in \omega _{w}(x_{n})\subseteq F(T)$, $\{d(x_{n},u)\}$
converges, so by Lemma \ref{(lemma2.11)} we have $x=u,$ this means that $%
\omega _{w}(x_{n})$ include exactly one point.
\end{proof}

\begin{theorem}
\label{(teo4.4)}Let $\kappa >0$ and $X$ be a complete $CAT(\kappa )$ space
with $diam(X)\leq \frac{\pi -\varepsilon }{2\sqrt{\kappa }}$ for some $%
\varepsilon \in (0,\pi /2)$. Let $K$ be a nonempty closed convex subset of $%
X $, and $T:K\rightarrow C(X)$ be a generalized multivalued hybrid mapping
type I with $\frac{2k_{1}(x)}{1-a_{2}(x)}<\frac{R}{2}$ for all $x\in K$
where $R=(\pi -2\varepsilon )\tan \varepsilon ),$ $F(T)\neq \emptyset $ and $%
Tp=\{p\}$ for all $p\in F(T)$. If $\{x_{n}\}$ is a sequence in $K$ \ defined
by (1.1) with $\lim \inf_{n\rightarrow \infty }\alpha _{n}[(1-\alpha _{n})%
\frac{R}{2}-\frac{2k_{2}(x)}{1-a_{3}(x)}]>0$ and $\lim \inf_{n\rightarrow
\infty }\beta _{n}[(1-\beta _{n})\frac{R}{2}-\frac{2k_{2}(x)}{1-a_{3}(x)}]>0$
then $\{x_{n}\}$ have a $\Delta -$limit which in $F(T).$
\end{theorem}

\begin{proof}
Let $p\in F(T)$ then for any $x\in K,u\in Tx$ we have that%
\begin{equation*}
d^{2}(u,p)\leq d^{2}(x,p)+\frac{k_{2}(p)}{1-a_{3}(p)}d^{2}(u,x)
\end{equation*}%
since metric projection $P_{K}$ is nonexpansive by Lemma \ref{(lemma2.12)},
and $P_{K}(p)=\{x\in K:d(p,x)=d(p,K)\}=\{p\}$ we have%
\begin{eqnarray*}
d^{2}(y_{n},p) &=&d^{2}(P_{K}((1-\beta _{n})x_{n}\oplus \beta
_{n}v_{n}),P_{K}(p)) \\
&\leq &d^{2}((1-\beta _{n})x_{n}\oplus \beta _{n}v_{n},p) \\
&\leq &(1-\beta _{n})d^{2}(x_{n},p)+\beta _{n}d^{2}(v_{n},p) \\
&&-\frac{R}{2}(1-\beta _{n})\beta _{n}d^{2}(x_{n},v_{n}) \\
&\leq &(1-\beta _{n})d^{2}(x_{n},p)+\beta _{n}[d^{2}(x,p)+\frac{k_{2}(p)}{%
1-a_{3}(p)}d^{2}(v_{n},x)] \\
&&-\frac{R}{2}(1-\beta _{n})\beta _{n}d^{2}(x_{n},v_{n}) \\
&\leq &(1-\beta _{n})d^{2}(x_{n},p)+\beta _{n}(d^{2}(x_{n},p) \\
&&+\frac{k_{2}(p)}{1-a_{3}(p)}d^{2}(v_{n},x_{n}))-\frac{R}{2}(1-\beta
_{n})\beta _{n}d^{2}(x_{n},v_{n}) \\
&\leq &d^{2}(x_{n},p)+\beta _{n}(\frac{k_{2}(p)}{1-a_{3}(p)}-\frac{R}{2}%
(1-\beta _{n}))d^{2}(v_{n},x_{n})) \\
&\leq &d^{2}(x_{n},p)
\end{eqnarray*}%
and 
\begin{eqnarray*}
d^{2}(x_{n+1},p) &=&d^{2}(P_{K}((1-\alpha _{n})y_{n}\oplus \alpha
_{n}u_{n}),P_{K}(p)) \\
&\leq &d^{2}((1-\alpha _{n})y_{n}\oplus \alpha _{n}u_{n}),p) \\
&\leq &(1-\alpha _{n})d^{2}(y_{n},p)+\alpha _{n}d^{2}(u_{n},p) \\
&&-\frac{R}{2}(1-\alpha _{n})\alpha _{n}d^{2}(y_{n},u_{n}) \\
&\leq &(1-\alpha _{n})d^{2}(y_{n},p)+\alpha _{n}[d^{2}(y,p)+\frac{k_{2}(p)}{%
1-a_{3}(p)}d^{2}(u_{n},y)] \\
&&-\frac{R}{2}(1-\alpha _{n})\alpha _{n}d^{2}(y_{n},u_{n}) \\
&\leq &(1-\alpha _{n})d^{2}(y_{n},p)+\alpha _{n}(d^{2}(y_{n},p) \\
&&+\frac{k_{2}(p)}{1-a_{3}(p)}d^{2}(u_{n},y_{n}))-\frac{R}{2}(1-\alpha
_{n})\alpha _{n}d^{2}(y_{n},u_{n}) \\
&\leq &(1-\alpha _{n})d^{2}(y_{n},p)+\alpha _{n}(d^{2}(y_{n},p)) \\
&&+\alpha _{n}(\frac{k_{2}(p)}{1-a_{3}(p)}-\frac{R}{2}(1-\alpha
_{n}))d^{2}(u_{n},y_{n})) \\
&\leq &(1-\alpha _{n})d^{2}(y_{n},p)+\alpha _{n}(d^{2}(y_{n},p)) \\
&&+\alpha _{n}(\frac{k_{2}(p)}{1-a_{3}(p)}-\frac{R}{2}(1-\alpha
_{n}))d^{2}(u_{n},y_{n})) \\
&\leq &d^{2}(y_{n},p)+\alpha _{n}(\frac{k_{2}(p)}{1-a_{3}(p)}-\frac{R}{2}%
(1-\alpha _{n}))d^{2}(u_{n},y_{n})) \\
&\leq &d^{2}(y_{n},p) \\
&\leq &d^{2}(x_{n},p).
\end{eqnarray*}

Here we have $d^{2}(x_{n+1},p)\leq d^{2}(x_{n},p)$ implies that $%
\lim_{n\rightarrow \infty }d(x_{n},p)$ exists,and $d(x_{n+1},p)\leq
d(y_{n},p)\leq d(x_{n},p)$ implies $\lim_{n\rightarrow \infty
}[d(x_{n},p)-d(y_{n},p)]=0$. Since $\beta _{n}(\frac{k_{2}(p)}{1-a_{3}(p)}-%
\frac{R}{2}(1-\beta _{n}))d^{2}(v_{n},x_{n}))\leq
d^{2}(x_{n},p)-d^{2}(y_{n},p),$by assumption we have that $%
\lim_{n\rightarrow \infty }d^{2}(v_{n},x_{n})=0,$so $\lim_{n\rightarrow
\infty }d(v_{n},x_{n})=0,$ $\lim_{n\rightarrow \infty }d(Tx_{n},x_{n})=0.$
Hence, by Lemma \ref{(lemma4.1)}, $\omega _{w}(x_{n})\subseteq F(T)$ and $%
\omega _{w}(x_{n})$ include exactly one point.,This means that $\{x_{n}\}$
have a $\Delta -$limit which in $F(T)$
\end{proof}

\begin{theorem}
\label{(teo4.6)}Let $\kappa >0$ and $X$ be a complete $CAT(\kappa )$ space
with $diam(X)\leq \frac{\pi -\varepsilon }{2\sqrt{\kappa }}$ for some $%
\varepsilon \in (0,\pi /2)$. Let $K$ be a nonempty compact convex subset of $%
X$, and $T:K\rightarrow C(X)$ be a continuous generalized multivalued hybrid
mapping type I with $\frac{2k_{1}(x)}{1-a_{2}(x)}<\frac{R}{2}$ for all $x\in
K$ where $R=(\pi -2\varepsilon )\tan \varepsilon ),$ $F(T)\neq \emptyset $
and $Tp=\{p\}$ for all $p\in F(T)$. If $\{x_{n}\}$ is a sequence in $K$ \
defined by (1.1) with $\lim \inf_{n\rightarrow \infty }\alpha _{n}[(1-\alpha
_{n})\frac{R}{2}-\frac{2k_{2}(x)}{1-a_{3}(x)}]>0$ and $\lim
\inf_{n\rightarrow \infty }\beta _{n}[(1-\beta _{n})\frac{R}{2}-\frac{%
2k_{2}(x)}{1-a_{3}(x)}]>0$ then $\{x_{n}\}$ have a $\Delta -$limit which in $%
F(T).$
\end{theorem}

\begin{proof}
By Theorem \ref{(teo4.4)}, we have that $\lim_{n\rightarrow \infty
}d(Tx_{n},x_{n})=0$ and $\lim_{n\rightarrow \infty }d(x_{n},p)$ exists for
all $p\in F(T)$\ Since $K$ is compact there is a convergent subsequence $%
\{x_{n_{i}}\}$ of $\{x_{n}\},$ say $\lim_{i\rightarrow \infty
}x_{_{n_{i}}}=z.$ Then we have

\begin{equation*}
d(z,Tz)\leq
d(z,x_{_{n_{i}}})+d(x_{_{n_{i}}},Tx_{_{n_{i}}})+H(Tx_{_{n_{i}}},Tz)
\end{equation*}%
and taking limit on $i,$countinuity of $T$ implies that $\ z\in Tz.$
\end{proof}

\begin{theorem}
\label{(teo4.5)}Let $\kappa >0$ and $X$ be a complete $CAT(\kappa )$ space
with $diam(X)\leq \frac{\pi -\varepsilon }{2\sqrt{\kappa }}$ for some $%
\varepsilon \in (0,\pi /2)$. Let $K$ be a nonempty closed convex subset of $%
X $, and $T:K\rightarrow C(X)$ be a $(a_{1},a_{2},b_{1},b_{2})-$multivalued
hybrid mapping type I with $T(p)=\{p\}$ for all $p\in F(T)$. Let$\{x_{n}\}$
be a sequence in $K$ defined by (1.2) have a $\Delta -$limit which in $F(T)$
\end{theorem}

\begin{proof}
Let $p\in F(T)$ then $T(p)=\{p\}$since $T$ is a $(a_{1},a_{2},b_{1},b_{2})-$%
multivalued hybrid mapping type I, for all $x\in K,u\in Tx$ we have that%
\begin{eqnarray*}
d^{2}(u,p) &\leq &a_{1}(x)d^{2}(u,p)+a_{2}(x)d^{2}(u,p) \\
&\leq &b_{1}(x)d^{2}(p,x)+b_{2}(x)d^{2}(p,x) \\
&\leq &d^{2}(p,x)
\end{eqnarray*}%
Hence we have that $d(u,p)\leq d(x,p).$ Then 
\begin{eqnarray*}
d(z_{n},p) &=&d(P_{K}((1-\beta _{n})x_{n}\oplus \beta _{n}w_{n}),P_{K}(p)) \\
&=&d(P_{K}((1-\beta _{n})x_{n}\oplus \beta _{n}w_{n}),p) \\
&\leq &d((1-\beta _{n})x_{n}\oplus \beta _{n}w_{n},p) \\
&\leq &(1-\beta _{n})d(x_{n},p)+\beta _{n}d(w_{n},p) \\
&\leq &(1-\beta _{n})d(x_{n},p)+\beta _{n}d(x_{n},p) \\
&\leq &(1-\beta _{n})d(x_{n},p)+\beta _{n}d(x_{n},p) \\
&\leq &d(x_{n},p)
\end{eqnarray*}%
and 
\begin{eqnarray*}
d(y_{n},p) &=&d(P_{K}((1-\alpha _{n})w_{n}\oplus \alpha _{n}v_{n},p) \\
&\leq &d(P_{K}((1-\alpha _{n})w_{n}\oplus \alpha _{n}v_{n},P_{K}(p)) \\
&\leq &d((1-\alpha _{n})w_{n}\oplus \alpha _{n}v_{n},p) \\
&\leq &(1-\alpha _{n})d(w_{n},p)+\alpha _{n}d(v_{n},p) \\
&\leq &(1-\alpha _{n})d(x_{n},p)+\alpha _{n}d(z_{n},p) \\
&\leq &(1-\alpha _{n})d(x_{n},p)+\alpha _{n}d(z_{n},p) \\
&\leq &d(x_{n},p)
\end{eqnarray*}%
and%
\begin{eqnarray*}
d(x_{n+1},p) &=&d(P_{K}(u_{n}),P_{K}(p)) \\
&\leq &d(u_{n},p) \\
&\leq &d(y_{n},p)
\end{eqnarray*}

so by $d(x_{n+1},p)\leq d(y_{n},p)\leq d(x_{n},p)$ implies $%
\lim_{n\rightarrow \infty }d(x_{n},p)=\lim_{n\rightarrow \infty }d(y_{n},p)$
exist.Let say $\lim_{n\rightarrow \infty }d(x_{n},p)=k.$ Since $%
d(w_{n},p)\leq d(x_{n},p)$ and $d(v_{n},p)\leq d(z_{n},p)\leq d(x_{n},p)$ we
have that $\limsup_{n\rightarrow \infty }d(w_{n},p)\leq
k,\limsup_{n\rightarrow \infty }d(v_{n},p)\leq k$ and 
\begin{eqnarray*}
d(y_{n},p) &=&d(P_{K}((1-\alpha _{n})w_{n}\oplus \alpha _{n}v_{n},p) \\
&\leq &d(P_{K}((1-\alpha _{n})w_{n}\oplus \alpha _{n}v_{n},P_{K}(p)) \\
&\leq &d((1-\alpha _{n})w_{n}\oplus \alpha _{n}v_{n},p) \\
&\leq &(1-\alpha _{n})d(w_{n},p)+\alpha _{n}d(v_{n},p) \\
&\leq &(1-\alpha _{n})d(x_{n},p)+\alpha _{n}d(z_{n},p) \\
&\leq &(1-\alpha _{n})d(x_{n},p)+\alpha _{n}d(z_{n},p) \\
&\leq &d(x_{n},p)
\end{eqnarray*}

implies that that $\lim_{n\rightarrow \infty }d((1-\alpha _{n})w_{n}+\alpha
_{n}v_{n},p)=k,$ so by Lemma \ref{(lemma2.22)} we have that $%
\lim_{n\rightarrow \infty }d(w_{n},v_{n})=0.$ And again from

\begin{eqnarray*}
d(y_{n},p) &=&d(P_{K}((1-\alpha _{n})w_{n}\oplus \alpha _{n}v_{n},p) \\
&\leq &d(P_{K}((1-\alpha _{n})w_{n}\oplus \alpha _{n}v_{n},P_{K}(p)) \\
&\leq &d((1-\alpha _{n})w_{n}\oplus \alpha _{n}v_{n},p) \\
&\leq &(1-\alpha _{n})d(w_{n},p)+\alpha _{n}d(v_{n},p) \\
&\leq &(1-\alpha _{n})(d(w_{n},v_{n})+d(v_{n},p))+\alpha _{n}d(v_{n},p) \\
&\leq &(1-\alpha _{n})d(w_{n},v_{n})+d(v_{n},p)
\end{eqnarray*}

we have that $k\leq \liminf_{n\rightarrow \infty }d(v_{n},p),$ and since $%
d(v_{n},p)\leq d(z_{n},p)\leq d(x_{n},p)$ we have that $\lim_{n\rightarrow
\infty }d(x_{n},p)=k.$ By CN$^{\ast }$ inequality we have

\begin{eqnarray*}
d^{2}(z_{n},p) &=&d^{2}(P_{K}((1-\beta _{n})x_{n}\oplus \beta
_{n}w_{n}),P_{K}(p)) \\
&=&d^{2}(P_{K}((1-\beta _{n})x_{n}\oplus \beta _{n}w_{n}),p) \\
&\leq &d^{2}((1-\beta _{n})x_{n}\oplus \beta _{n}w_{n},p) \\
&\leq &(1-\beta _{n})d^{2}(x_{n},p)+\beta _{n}d^{2}(w_{n},p)-\frac{R}{2}%
(1-\beta _{n})\beta _{n}d^{2}(x_{n},w_{n}) \\
&\leq &(1-\beta _{n})d^{2}(x_{n},p)+\beta _{n}d^{2}(x_{n},p)-\frac{R}{2}%
(1-\beta _{n})\beta _{n}d^{2}(x_{n},w_{n}) \\
&\leq &d^{2}(x_{n},p)-\frac{R}{2}(1-\beta _{n})\beta _{n}d^{2}(x_{n},w_{n})
\end{eqnarray*}

implies that 
\begin{equation*}
\frac{R}{2}(1-\beta _{n})\beta _{n}d^{2}(x_{n},w_{n})\leq
d^{2}(x_{n},p)-d^{2}(z_{n},p).
\end{equation*}

Since $\lim_{n\rightarrow \infty }(d^{2}(x_{n},p)-d^{2}(z_{n},p))=0$ and $%
\inf_{n}(1-\beta _{n})\beta _{n}>0$ then we have that $\lim_{n\rightarrow
\infty }d(x_{n},w_{n})=0$ and so$.$ $\lim_{n\rightarrow \infty
}d(x_{n},Tx_{n})=0.$ Then by Lemma \ref{(lemma4.2)} $\{x_{n}\}$have $\Delta
- $limit which in $F(T)$.
\end{proof}

\begin{theorem}
Let $\kappa >0$ and $X$ be a complete $CAT(\kappa )$ space with $diam(X)\leq 
\frac{\pi -\varepsilon }{2\sqrt{\kappa }}$ for some $\varepsilon \in (0,\pi
/2)$. Let $K$ be a nonempty compact convex subset of $X$, and $%
T:K\rightarrow KC(X)$ be a continuous $(a_{1},a_{2},b_{1},b_{2})-$%
multivalued hybrid mapping type I with $T(p)=\{p\}$ for all $p\in F(T)$. Let$%
\{x_{n}\}$ be a sequence in $K$ defined by (1.2) have a $\Delta -$limit
which in $F(T)$
\end{theorem}

\begin{proof}
Proof is similar to Theorem \ref{(teo4.6)}$.$
\end{proof}

\begin{theorem}
Let $\kappa >0$ and $X$ be a complete $CAT(\kappa )$ space with $diam(X)\leq 
\frac{\pi -\varepsilon }{2\sqrt{\kappa }}$ for some $\varepsilon \in (0,\pi
/2)$. Let $K$ be a nonempty closed convex subset of $X$, and $T:K\rightarrow
KC(X)$ be a generalized multivalued hybrid mapping type II with $\frac{%
2k_{1}(x)}{1-a_{2}(x)}<\frac{R}{2}$ for all $x\in K$ where $R=(\pi
-2\varepsilon )\tan \varepsilon ),$ $F(T)\neq \emptyset $ and $Tp=\{p\}$ for
all $p\in F(T)$. If $\{x_{n}\}$ is a sequence in $K$ \ defined by (1.1) with 
$\lim \inf_{n\rightarrow \infty }\alpha _{n}[(1-\alpha _{n})\frac{R}{2}-%
\frac{2k_{2}(x)}{1-a_{3}(x)}]>0$ and $\lim \inf_{n\rightarrow \infty }\beta
_{n}[(1-\beta _{n})\frac{R}{2}-\frac{2k_{2}(x)}{1-a_{3}(x)}]>0$ then $%
\{x_{n}\}$ have a $\Delta -$limit which in $F(T)$
\end{theorem}

\begin{proof}
Let $p\in F(T)$ then for any $x\in K,$we have that%
\begin{equation*}
H^{2}(Tx,Tp)\leq d^{2}(x,p)+\frac{k_{2}(p)}{1-a_{3}(p)}d^{2}(Tx,x)
\end{equation*}%
since metric projection $P_{K}$ is nonexpansive by Lemma \ref{(lemma2.12)},
and $P_{K}(p)=\{x\in K:d(p,x)=d(p,K)\}=\{p\}$ we have%
\begin{eqnarray*}
d^{2}(y_{n},p) &=&d^{2}(P_{K}((1-\beta _{n})x_{n}\oplus \beta
_{n}v_{n}),P_{K}(p)) \\
&\leq &d^{2}((1-\beta _{n})x_{n}\oplus \beta _{n}v_{n},p) \\
&\leq &(1-\beta _{n})d^{2}(x_{n},p)+\beta _{n}d^{2}(v_{n},p) \\
&&-\frac{R}{2}(1-\beta _{n})\beta _{n}d^{2}(x_{n},v_{n}) \\
&\leq &(1-\beta _{n})d^{2}(x_{n},p)+\beta _{n}d^{2}(v_{n},Tp) \\
&&-\frac{R}{2}(1-\beta _{n})\beta _{n}d^{2}(x_{n},Tx_{n}) \\
&\leq &(1-\beta _{n})d^{2}(x_{n},p)+\beta _{n}H^{2}(Tx_{n},Tp) \\
&&-\frac{R}{2}(1-\beta _{n})\beta _{n}d^{2}(x_{n},Tx_{n}) \\
&\leq &(1-\beta _{n})d^{2}(x_{n},p)+\beta _{n}(d^{2}(x_{n},p) \\
&&+\frac{k_{2}(p)}{1-a_{3}(p)}d^{2}(Tx_{n},x_{n}))-\frac{R}{2}(1-\beta
_{n})\beta _{n}d^{2}(x_{n},Tx_{n}) \\
&\leq &d^{2}(x_{n},p)+\beta _{n}(\frac{k_{2}(p)}{1-a_{3}(p)}-\frac{R}{2}%
(1-\beta _{n}))d^{2}(Tx_{n},x_{n})) \\
&\leq &d^{2}(x_{n},p)
\end{eqnarray*}%
and 
\begin{eqnarray*}
d^{2}(x_{n+1},p) &=&d^{2}(P_{K}((1-\alpha _{n})y_{n}\oplus \alpha
_{n}u_{n}),P_{K}(p)) \\
&\leq &d^{2}((1-\alpha _{n})y_{n}\oplus \alpha _{n}u_{n}),p) \\
&\leq &(1-\alpha _{n})d^{2}(y_{n},p)+\alpha _{n}d^{2}(u_{n},p) \\
&&-\frac{R}{2}(1-\alpha _{n})\alpha _{n}d^{2}(y_{n},u_{n}) \\
&\leq &(1-\alpha _{n})d^{2}(y_{n},p)+\alpha _{n}d^{2}(u_{n},Tp) \\
&&-\frac{R}{2}(1-\alpha _{n})\alpha _{n}d^{2}(y_{n},Ty_{n}) \\
&\leq &(1-\alpha _{n})d^{2}(y_{n},p)+\alpha _{n}H^{2}(Ty_{n},Tp) \\
&&-\frac{R}{2}(1-\alpha _{n})\alpha _{n}d^{2}(y_{n},Ty_{n}) \\
&\leq &(1-\alpha _{n})d^{2}(y_{n},p)+\alpha _{n}(d^{2}(y_{n},p) \\
&&+\frac{k_{2}(p)}{1-a_{3}(p)}d^{2}(Ty_{n},y_{n}))-\frac{R}{2}(1-\alpha
_{n})\alpha _{n}d^{2}(y_{n},Ty_{n}) \\
&\leq &(1-\alpha _{n})d^{2}(y_{n},p)+\alpha _{n}(d^{2}(y_{n},p)) \\
&&+\alpha _{n}(\frac{k_{2}(p)}{1-a_{3}(p)}-\frac{R}{2}(1-\alpha
_{n}))d^{2}(Ty_{n},y_{n})) \\
&\leq &(1-\alpha _{n})d^{2}(y_{n},p)+\alpha _{n}(d^{2}(y_{n},p)) \\
&&+\alpha _{n}(\frac{k_{2}(p)}{1-a_{3}(p)}-\frac{R}{2}(1-\alpha
_{n}))d^{2}(Ty_{n},y_{n})) \\
&\leq &d^{2}(y_{n},p)+\alpha _{n}(\frac{k_{2}(p)}{1-a_{3}(p)}-\frac{R}{2}%
(1-\alpha _{n}))d^{2}(Ty_{n},y_{n})) \\
&\leq &d^{2}(y_{n},p) \\
&\leq &d^{2}(x_{n},p).
\end{eqnarray*}

Here we have $d^{2}(x_{n+1},p)\leq d^{2}(x_{n},p)$ implies that $%
\lim_{n\rightarrow \infty }d(x_{n},p)$ exists,and $d(x_{n+1},p)\leq
d(y_{n},p)\leq d(x_{n},p)$ implies $\lim_{n\rightarrow \infty
}[d(x_{n},p)-d(y_{n},p)]=0$. Since $\beta _{n}(\frac{k_{2}(p)}{1-a_{3}(p)}-%
\frac{R}{2}(1-\beta _{n}))d^{2}(Tx_{n},x_{n}))\leq
d^{2}(x_{n},p)-d^{2}(y_{n},p),$by assumption we have that $%
\lim_{n\rightarrow \infty }d^{2}(Tx_{n},x_{n})=0,$so $\lim_{n\rightarrow
\infty }d(Tx_{n},x_{n})=0$ Hence, by Lemma \ref{(lemma4.3)}, $\omega
_{w}(x_{n})\subseteq F(T)$ and $\omega _{w}(x_{n})$ include exactly one
point.,This means that $\{x_{n}\}$ have a $\Delta -$limit which in $F(T)$
\end{proof}

\begin{theorem}
LLet $\kappa >0$ and $X$ be a complete $CAT(\kappa )$ space with $%
diam(X)\leq \frac{\pi -\varepsilon }{2\sqrt{\kappa }}$ for some $\varepsilon
\in (0,\pi /2)$. Let $K$ be a nonempty compact convex subset of $X$, and $%
T:K\rightarrow KC(X)$ be a generalized multivalued hybrid mapping type II
with $\frac{2k_{1}(x)}{1-a_{2}(x)}<\frac{R}{2}$ for all $x\in K$ where $%
R=(\pi -2\varepsilon )\tan \varepsilon ),$ $F(T)\neq \emptyset $ and $%
Tp=\{p\}$ for all $p\in F(T)$. If $\{x_{n}\}$ is a sequence in $K$ \ defined
by (1.1) with $\lim \inf_{n\rightarrow \infty }\alpha _{n}[(1-\alpha _{n})%
\frac{R}{2}-\frac{2k_{2}(x)}{1-a_{3}(x)}]>0$ and $\lim \inf_{n\rightarrow
\infty }\beta _{n}[(1-\beta _{n})\frac{R}{2}-\frac{2k_{2}(x)}{1-a_{3}(x)}]>0$
then $\{x_{n}\}$ have a $\Delta -$limit which in $F(T)$
\end{theorem}

\begin{proof}
Proof is similar to Theorem \ref{(teo4.6)}$.$
\end{proof}

\begin{theorem}
Let $\kappa >0$ and $X$ be a complete $CAT(\kappa )$ space with $diam(X)\leq 
\frac{\pi -\varepsilon }{2\sqrt{\kappa }}$ for some $\varepsilon \in (0,\pi
/2)$. Let $K$ be a nonempty closed convex subset of $X$, and $T:K\rightarrow
KC(X)$ be a $(a_{1},a_{2},b_{1},b_{2})-$multivalued hybrid mapping type II
with $F(T)\neq \emptyset $, $Tp=\{p\}$ for all $p\in F(T)$ and $a_{1}(x)\geq
1$ for all $x\in C$ . Let$\{x_{n}\}$ be a sequence in $K$ defined by (1.2)
strongly converges to a point of $F(T)$
\end{theorem}

\begin{proof}
Let $p\in F(T)$ then $T(p)=\{p\}$since $T$ is a $(a_{1},a_{2},b_{1},b_{2})-$%
multivalued hybrid mapping we have that%
\begin{eqnarray*}
d^{2}(Tx,p) &\leq &a_{1}(x)d^{2}(Tx,p)+a_{2}(p)d^{2}(Tx,p) \\
&\leq &a_{1}(x)H^{2}(Tx,Tp)+a_{2}(x)d^{2}(Tx,p) \\
&\leq &b_{1}(x)d^{2}(x,p)+b_{2}(x)d^{2}(x,p) \\
&\leq &b_{1}(x)d^{2}(x,p)+b_{2}(x)d^{2}(x,p) \\
&\leq &d^{2}(x,p)
\end{eqnarray*}%
Assume that $a_{2}(x)\geq 0$ for all $x\in C$ since $T$ is a $%
(a_{1},a_{2},b_{1},b_{2})-$multivalued hybrid mapping we have that 
\begin{eqnarray*}
a_{1}(x)H^{2}(Tx,Tp) &\leq
&b_{1}(x)d^{2}(x,p)+b_{2}(x)d^{2}(x,p)-a_{2}(x)d^{2}(Tx,p) \\
&\leq &d^{2}(x,p)
\end{eqnarray*}%
so we have%
\begin{equation*}
H^{2}(Tx,Tp)\leq \frac{1}{a_{1}(x)}d^{2}(x,p)\leq d^{2}(x,p)
\end{equation*}%
And now if $a_{2}(x)\leq 0$ for all $x\in C$ then since $a_{1}(x)+a_{2}(x)%
\geq 1$, $1\geq \frac{1}{a_{1}(x)}-\frac{a_{2}(x)}{a_{1}(x)}$ and since $T$
is a $(a_{1},a_{2},b_{1},b_{2})-$multivalued hybrid mapping we have that%
\begin{eqnarray*}
a_{1}(x)H^{2}(Tx,Tp) &\leq
&b_{1}(x)d^{2}(x,p)+b_{2}(x)d^{2}(x,p)-a_{2}(x)d^{2}(Tx,p) \\
&\leq &b_{1}(x)d^{2}(x,p)+b_{2}(x)d^{2}(x,p)-a_{2}(x)d^{2}(Tx,p) \\
&\leq &d^{2}(x,p)-a_{2}(x)d^{2}(Tx,p)
\end{eqnarray*}%
which implies that 
\begin{eqnarray*}
H^{2}(Tx,Tp) &\leq &\frac{1}{a_{1}(x)}d^{2}(x,p)-\frac{a_{2}(p)}{a_{1}(p)}%
d^{2}(Tx,p) \\
&\leq &\frac{1}{a_{1}(x)}d^{2}(x,p)-\frac{a_{2}(x)}{a_{1}(x)}d^{2}(x,p) \\
&=&(\frac{1}{a_{1}(x)}-\frac{a_{2}(x)}{a_{1}(x)})d^{2}(p,x) \\
&\leq &d^{2}(p,x).
\end{eqnarray*}

Hence we have that $H(Tp,Tx)\leq d(p,x)$.%
\begin{eqnarray*}
d(z_{n},p) &=&d(P_{K}((1-\beta _{n})x_{n}\oplus \beta _{n}w_{n}),P_{K}(p)) \\
&=&d(P_{K}((1-\beta _{n})x_{n}\oplus \beta _{n}w_{n}),p) \\
&\leq &d((1-\beta _{n})x_{n}\oplus \beta _{n}w_{n},p) \\
&\leq &(1-\beta _{n})d(x_{n},p)+\beta _{n}d(w_{n},p) \\
&\leq &(1-\beta _{n})d(x_{n},p)+\beta _{n}d(w_{n},Tp) \\
&\leq &(1-\beta _{n})d(x_{n},p)+\beta _{n}H(Tx_{n},Tp) \\
&\leq &(1-\beta _{n})d(x_{n},p)+\beta _{n}d(x_{n},p) \\
&\leq &d(x_{n},p)
\end{eqnarray*}%
and 
\begin{eqnarray*}
d(y_{n},p) &=&d(P_{K}((1-\alpha _{n})w_{n}\oplus \alpha _{n}v_{n},p) \\
&\leq &d(P_{K}((1-\alpha _{n})w_{n}\oplus \alpha _{n}v_{n},P_{K}(p)) \\
&\leq &d((1-\alpha _{n})w_{n}\oplus \alpha _{n}v_{n},p) \\
&\leq &d((1-\alpha _{n})w_{n}\oplus \alpha _{n}v_{n}),p) \\
&\leq &(1-\alpha _{n})d(w_{n},p)+\alpha _{n}d(v_{n},p) \\
&\leq &(1-\alpha _{n})d(w_{n},Tp)+\alpha _{n}d(v_{n},Tp) \\
&\leq &(1-\alpha _{n})H(Tx_{n},p)+\alpha _{n}H(Tz_{n},Tp) \\
&\leq &(1-\alpha _{n})d(x_{n},p)+\alpha _{n}d(z_{n},p) \\
&\leq &d^{2}(z_{n},p)
\end{eqnarray*}%
and%
\begin{eqnarray*}
d^{2}(x_{n+1},p) &=&d(P_{K}(u_{n}),P_{K}(p)) \\
&\leq &d(u_{n},p) \\
&\leq &H(Ty_{n},Tp) \\
&\leq &d(y_{n},p)
\end{eqnarray*}

so by $d(x_{n+1},p)\leq d(y_{n},p)\leq d(z_{n},p)\leq d(x_{n},p)$ implies $%
\lim_{n\rightarrow \infty }d(x_{n},p)$ exist. By CN$^{\ast }$ inequality we
have

\begin{eqnarray*}
d^{2}(z_{n},p) &=&d^{2}(P_{K}((1-\beta _{n})x_{n}\oplus \beta
_{n}w_{n}),P_{K}(p)) \\
&=&d^{2}(P_{K}((1-\beta _{n})x_{n}\oplus \beta _{n}w_{n}),p) \\
&\leq &d^{2}((1-\beta _{n})x_{n}\oplus \beta _{n}w_{n},p) \\
&\leq &(1-\beta _{n})d^{2}(x_{n},p)+\beta _{n}d^{2}(w_{n},p)-\frac{R}{2}%
(1-\beta _{n})\beta _{n}d^{2}(x_{n},w_{n}) \\
&\leq &(1-\beta _{n})d^{2}(x_{n},p)+\beta _{n}d^{2}(w_{n},Tp)-\frac{R}{2}%
(1-\beta _{n})\beta _{n}d^{2}(x_{n},Tx_{n}) \\
&\leq &(1-\beta _{n})d^{2}(x_{n},p)+\beta _{n}H^{2}(Tx_{n},Tp)-\frac{R}{2}%
(1-\beta _{n})\beta _{n}d^{2}(x_{n},Tx_{n}) \\
&\leq &(1-\beta _{n})d^{2}(x_{n},p)+\beta _{n}d^{2}(x_{n},p)-\frac{R}{2}%
(1-\beta _{n})\beta _{n}d^{2}(x_{n},Tx_{n}) \\
&\leq &d^{2}(x_{n},p)-\frac{R}{2}(1-\beta _{n})\beta _{n}d^{2}(x_{n},Tx_{n})
\end{eqnarray*}

implies that 
\begin{equation*}
\frac{R}{2}(1-\beta _{n})\beta _{n}d^{2}(x_{n},Tx_{n})\leq
d^{2}(x_{n},p)-d^{2}(z_{n},p).
\end{equation*}

Since $\lim_{n\rightarrow \infty }(d(x_{n},p)-d^{2}(z_{n},p))=0$ and $%
\inf_{n}(1-\beta _{n})\beta _{n}>0$ then we have that $\lim_{n\rightarrow
\infty }d(x_{n},Tx_{n})=0.$ Rest of th proof is similar to Theorem \ref%
{(teo4.6)}$.$

Then by Lemma \ref{(teo4.2)} $\{x_{n}\}$have $\Delta -$limit which in $F(T).$
The above discussion will hold if $a_{2}(p)\geq 0$
\end{proof}


\end{document}